\documentclass[a4paper,11pt,twoside]{article}

\RequirePackage[colorlinks,citecolor=blue,urlcolor=blue]{hyperref}
\usepackage[ruled,vlined]{algorithm2e}
\usepackage[utf8]{inputenc}
\usepackage[T1]{fontenc}
\usepackage[english]{babel}
\usepackage{soul} 
\usepackage{charter}
\usepackage{enumerate}
\usepackage{indentfirst}
\usepackage{xcolor}
\usepackage[top=5.cm, bottom=5.0cm, left=3cm, right=3cm]{geometry}
\usepackage{authblk}
\usepackage{mathtools}

\usepackage{graphicx}

\usepackage{amsmath}
\usepackage{amsthm}	
\usepackage{amsfonts}
\usepackage{subcaption}	
\usepackage{amssymb}

\theoremstyle{plain}
\newtheorem{theorem}{Theorem}[section]
\newtheorem{lemma}[theorem]{Lemma}

\newtheorem{assumption}[theorem]{Assumption}

\theoremstyle{definition}

\numberwithin{equation}{section}
\numberwithin{figure}{section}
 \usepackage[nodayofweek]{datetime}

\usepackage[nodayofweek]{datetime}

\newcommand{\R}{\mathbb{R}}	
\newcommand{\F}{\mathcal{F}}	
\newcommand{\norm}[1]{\lVert#1\rVert}	

\newcommand{\bE}{\mathbb{E}}

\newcommand{\bN}{\mathbb{N}}
\newcommand{\bP}{\mathbb{P}}

\newcommand{\bR}{\mathbb{R}}

\newcommand{\cF}{\mathcal{F}}
\newcommand{\cG}{\mathcal{G}}

\newcommand{\cM}{\mathcal{M}}

\newcommand{\cO}{\mathcal{O}}
\newcommand{\cP}{\mathcal{P}}

\newcommand{\cV}{\mathcal{V}}
\newcommand{\cX}{\mathcal{X}}

\def\t{\tau}

\usepackage[colorinlistoftodos,bordercolor=orange,backgroundcolor=orange!20,linecolor=orange,textsize=scriptsize]{todonotes}
 
\newcommand{\1}{\mathbbm{1}}

\newcommand{\hLip}{\textbf{(Lip)}{ }}
\newcommand{\hkreg}{\textbf{(Ker-Reg)}{ }}

\newcommand{\hvreg}{\textbf{($\cV$-Reg)}{ }}
\newcommand{\hvb}{\textbf{($\cV$-bound)}{ }}

\newcommand{\hvdiffs}{\textbf{(v-diff-Reg+)}}

\newcommand{\hvLip}{\textbf{($\cV$-Lip)}}

\newcommand{\hLaw}{\textbf{($\mu_0$-$L_p$)}{ }}

\newcommand{\E}{\mathbb{E}}

\newcommand{\RR}{\mathbb{R}}
\newcommand{\NN}{\mathbb{N}}

\def \F {\mathcal{F}}

\def \eps {\epsilon}

\def \< {\langle}  \def \> {\rangle}

  \def\t{\tau}

\def\1{\mathbf{1}}

\newcommand{\tr}{^T}
\newcommand{\smsec}{\sum_{i,j=1}^d}
\newcommand{\smfir}{\sum_{j=1}^d}
\newcommand{\smfi}{\sum_{i=1}^r}

\newcommand{\Rd}{{\R^d}}

\newcommand{\defeq}{\vcentcolon=}

\newcommand{\ells}{{\ell^{*}}}

\newcommand{\rom}[1]{\uppercase\expandafter{\romannumeral #1\relax}}

\title{Iterative Multilevel Particle Approximation for McKean-Vlasov SDEs}
\author[1,2]{Lukasz Szpruch}
\author[1]{Shuren Tan}
\author[1]{Alvin Tse}
\affil[1]{School of Mathematics,University of Edinburgh  }
\affil[2]{The Alan Turing Institute, London}

\date{ }

\begin{document}
\selectlanguage{english}
\maketitle

\begin{abstract}
The mean field limits of systems of interacting diffusions (also called stochastic interacting particle systems (SIPS)) have been intensively studied since McKean \cite{mckean1966class} as they pave a way to probabilistic representations for many important nonlinear/nonlocal PDEs. The fact that particles are not independent render classical variance reduction techniques not directly applicable and consequently make simulations of interacting diffusions prohibitive.  

In this article, we provide an alternative iterative particle representation, inspired by the fixed point argument by Sznitman \cite{sznitman1991topics}. The representation enjoys suitable conditional independence property that is leveraged in our analysis. We establish weak convergence of iterative particle system to the McKean-Vlasov SDEs (McKV-SDEs).  One of the immediate advantages of iterative particle system is that it can be combined with the Multilevel Monte Carlo (MLMC) approach for the simulation of McKV-SDEs. We proved that the MLMC approach reduces the computational complexity of calculating expectations by an order of magnitude.
Another perspective on this work is that we analyse the error of 
  nested Multilevel Monte Carlo estimators, which is of independent interest. Furthermore, we work with state dependent functionals, unlike scalar outputs which are common in literature on MLMC. The error analysis is carried out in uniform, and what seems to be new, weighted norms.
\end{abstract}

; secondary: 
60H30
.\\



\section{Introduction}

%

The theory of mean field interacting particle systems was pioneered by the work of H.~McKean \cite{mckean1966class}, where he gave a probabilistic interpretation of a class of nonlinear (due to the dependence on the coefficients of the solution itself) nonlocal PDEs. Probabilistic representation has an advantage, as it paves a way to Monte-Carlo approximation methods which are efficient 
in high dimensions.  
Fix $T>0$. Let $\{W_t\}_{t\in [0,T]}$ be an $r$-dimensional Brownian motion on a filtered probability space $(\Omega, \{ \cF_t \}_{t}, \mathcal{F}, \bP)$. Consider continuous functions $b: \mathbb{R}^{d} \times \mathbb{R}^{d}  \rightarrow \mathbb{R}^d$, $\sigma: \mathbb{R}^{d} \times \mathbb{R}^{d} \rightarrow \mathbb{R}^{d \otimes r}$ and their corresponding non-linear (in the sense of McKean) stochastic differential equation (McKV-SDE)
given by
\begin{equation}
\label{eq:general1}
\left\{
\begin{array} {l l}
	dX_t & = 	 b[X_t, \mu^X_t ] \,   dt+   \sigma[X_t, \mu^X_t ]  \, dW_t,  \\
	\mu^X_t &=  \text{Law}(X_t), \quad t\in[0,T],
\end{array}  \right.
\end{equation}
where $X_0 \sim \mu_0 \in \cP_2( \bR^d)$ and $G[x, m] := \int_{\bR^d} G(x,y)\, m(dy)$, for any $x \in \bR^d$ and $m \in \cP_2( \bR^d)$ (square-integrable laws on $\bR^d$). Notice that $\{X_t\}_{t\in [0,T]}$ is not necessarily a Markov process and hence it is not immediate what the corresponding backward Kolmogorov equation looks like. Nonetheless using It\^{o}'s formula with $P \in C^2_b(\bR^d)$,  one can derive corresponding nonlinear Kolmogorov-Fokker-Planck equation 
\begin{equation} \label{eq:pde}
   \partial_t \langle \mu_t, P \rangle   = 
   \langle \mu_t, \frac{1}{2}  \sum_{i,j=1}^d  \partial^2_{x_i, x_j}  P (\cdot)   \big( \sigma \sigma^T \big)_{ij}[\cdot, \mu_t] + \sum_{i=1}^d \partial_{x_i} P(\cdot)  b_i[\cdot, \mu_t ] \rangle, 
\end{equation} 
where $\langle m,F\rangle := \int_{\bR^d} F(y) \, m(dy)$, \cite{antonelli2002rate,chassagneuxprobabilistic,sznitman1991topics}. The theory of propagation of chaos, \cite{sznitman1991topics}, states that \eqref{eq:general1} arises as a limiting equation  of the system of interacting diffusions $\{ Y^{i,N}_t \}_{i=1, \ldots, N}$ on $(\bR^d)^N$ given by  
\begin{equation}
\label{eq:par} 
\left\{
  \begin{array}{l l}
d Y^{i,N}_t &=   b[Y^{i,N}_t,\mu^{Y,N}_t]   dt +  \sigma[Y^{i,N}_t,\mu^{Y,N}_t]  dW^i_t, 
\\
\mu^{Y,N}_t &:= \frac{1}{N} \sum_{i=1}^{N} \delta_{Y_t^{i,N}}, \quad t\geq0,
  \end{array} \right.	
\end{equation}
where $\{Y^{i,N}_0\}_{i=1,\ldots,N}$ are i.i.d samples with law $\mu_0$ and $\{W^i_t\}_{i=1,\ldots,N}$ are independent Brownian motions. It can be shown, under sufficient regularity conditions on the coefficients, that   $\mu^{Y,N}\in\cP_2(C([0,T],\bR^d))$  converges in law to $\mu^X$,  see \cite{meleard1996asymptotic}. This is a not trivial result as the particles are not independent. Moreover, \eqref{eq:par} can be interpreted as a first step towards numerical schemes for \eqref{eq:general1}.
To obtain a fully implementable algorithm  one needs to study time discretisation of \eqref{eq:general1}. As in seminal papers by Bossy and Talay \cite{bossy1996convergence,bossy1997stochastic} we work with an Euler scheme. Take partition $\{t_k\}_k$ of $[0,T]$, with $t_k-t_{k-1}=h$ and define $\eta{(t)}\defeq t_k\, \text{if}~t\in[t_k,t_{k+1})$. The continuous Euler scheme reads 
\begin{align} \label{eq:sEuler2}
\overline{Y}^{i,N}_{t}  &= \overline{Y}^{i,N}_{t_{k}} +  b[\overline{Y}^{i,N}_{\eta{(t)}},\mu^{\overline Y,N}_{\eta{(t)}} ] (t-t_k)
+ \sigma[\overline{Y}^{i,N}_{\eta{(t)}},\mu^{\overline Y,N}_{\eta{(t)}} ]  (  W^i_{t} - W^i_{{t_{k}}})\,. 
\end{align} 
Note that due to interactions between discretised diffusions, implementation  of \eqref{eq:sEuler2} requires  $N^2$ arithmetic operations at each step  $t_k$ of the scheme. This makes simulations of \eqref{eq:sEuler2} very costly, but should not come as a surprise as the aim is to approximate non linear/non local PDEs \eqref{eq:pde} for which the deterministic schemes based on space discretisation, typically, are also computationally very demanding \cite{bossy1997comparison}.   
It has been  proven that the empirical distribution function of $N$ particles \eqref{eq:sEuler2} converges, in a weak sense, to the distribution of the corresponding McKean-Vlasov limiting equation 
with the rate $O((\sqrt{N})^{-1} + {h})$, see \cite{antonelli2002rate, bossy2004optimal,bossy2002rate,bossy1997stochastic}. Hence the computational cost of achieving a mean-square-error (see Theorem \ref{lm:mseboundfinalinteractparticle} for the definition) of order $\epsilon^2>0$ using this direct approach is $\cO(\epsilon^{-5})$.

The lack of independence among interacting diffusions and the fact that the statistical error coming from approximating a measure creates a bias in the approximation, render applications of variance reduction techniques non-trivial. In fact, we are not aware of any rigorous work on variance reduction techniques for McKV-SDEs. In this article, we develop an iterated particle system  that allows decomposing the statistical error and bias. We also provide an error analysis for a general class of McKV-SDEs. Finally, we deploy the MLMC method of Giles-Heinrich \cite{giles2008multilevel,heinrich2001multilevel} (see also 2-level MC of Kebaier \cite{kebaier2005statistical}). In Section \ref{sec issues}, we show that a direct application of MLMC to \eqref{eq:par} fails.  It is worth  pointing  out that the idea of combining an iterative method with MLMC to solve non-linear PDEs has very recently been proposed in 
\cite{hutzenthaler2016full}. However, their interest is on BSDEs and their connections to semi-linear PDEs. 

The key technical part of the paper is weak convergence analysis of the time discretisation that allows for iteration of the error in a suitable norms. It is well know, at least since the work \cite{talay1990expansion} that weak error analysis relies on the corresponding PDE theory. However as we already stated the solution to \eqref{eq:general1} is not Markovian on $\bR^d$. To overcome we work with forward backward system
\begin{equation*}
\left\{
\begin{array} {l l}
		X^{0,X_0}_t  &= \xi + 	\int_0^t  b[X^{s,\xi}_s,\mu^{X^{0,\xi}}_s]   \,   ds +  \int_0^t  \sigma[X^{0,X_0}_s,\mu^{X^{0,X_0}}_s] \, dW_s,    \\
	\mu^{X^{0,X_0}}_t &=  \text{Law}(X^{0,X_0}_t), 
\end{array}  \right. 
\end{equation*}
and note that $X^{0,X_0}_t \neq X^{0,x}_t \big|_{x= X_0}  $ in general (see \cite{buckdahn2017mean}). This makes building of standard PDE theory on $[0,T]\times \bR^d$ problematic and lead to theory of PDEs on measure spaces proposed by P. Lions in his lectures in Coll\`{e}ge de France (\cite{lions2014cours}) and further developed in \cite{buckdahn2017mean,chassagneuxprobabilistic}. Here we work with
\begin{align}\label{eq:rewritextinteract}
 \cX_t^{0,x} = x + \int_0^t  b[\cX_s^{0,x}, \mu_s^{X^{0,\xi}}] \,d s +  \int_0^t  \sigma[\cX_s^{0,x}, \mu_s^{X^{0,\xi}}] dW_s.
\end{align}
Notice that \eqref{eq:rewritextinteract}, unlike \eqref{eq:general1}, is a Markov process. Furthermore, if \eqref{eq:general1} has a unique (weak) solution, then  $\cX_t^{0,x}|_{x=X_0} = X_t^{0,X_0}$. This means that
\[
\int_{\R^d}\E \big[ P(\cX^{0,x}_t) \big] \, \mu_0(dx)  =  \E \big[ \E[P(X_t) | X_0  ] \big]. 
\]
It can be shown that $v(0,x)=\E \big[ P(\cX^{0,x}_t) \big]$ is a solution to backward Kolmogorov equation on $[0,T]\times \bR^d$ which we will explore in this paper.

\subsection{Iterated particle method}
The main idea is to approximate \eqref{eq:general1} with a sequence of classical SDEs defined as 
\begin{gather}
\label{eq:defxtminteract}
dX^m_t  =b[X^m_t,\mu_{t}^{X^{m-1}}] dt + \sigma[X^m_t,\mu_{t}^{X^{m-1}}] dW^m_t, \quad \mu^{X^m}_0 = \mu^{X}_0, 
\end{gather}
where $(W^m, X^m_0)$ are independent for all $m \in \bN$ as well as $(W^m,X^m_0)$ and  $(W^n,X^n_0)$  $m\neq n \in \bN$, are independent. The conditional independence across iterations is the key difference of our approach from the proof of existence of solutions by Sznitman \cite{sznitman1991topics}, where the same Brownian motion and initial condition are used at every iteration.  
  The Euler scheme with $\mu^{\overline{X}^{m}}_0 = \mu^X_0$ reads 
 \begin{gather}\label{eq:xminteract}
	d\overline{X}^{m}_t = b[\overline{X}^{m}_{\eta(t)}, \mu_{\eta(t)}^{\overline{X}^{m-1}}] dt 
	+ \sigma[\overline{X}^{m}_{\eta(t)}, \mu_{\eta(t)}^{\overline{X}^{m-1}}] dW^{m}_t\,.
\end{gather}
To implement \eqref{eq:xminteract} at every step of the scheme, one needs to compute the integral with respect to the measure from the previous iteration $m-1$. This integral is calculated by approximating measure $\mu_{\eta (t)}^{\overline{X}^{m-1}}$ by the empirical measure with $N_{m-1}$ samples.
Consequently, we take $\mu^{{\overline{Y}}^{i,m}}_0 = \mu^X_0$ and define, for $m \in \bN$ and $1 \leq i \leq N_{m}$,
\begin{equation}\label{eq:xminteract2}
	d{\overline{Y}}^{i,m}_t =  b[{\overline{Y}}^{i,m}_{\eta(t)},\mu_{\eta (t)}^{\overline{Y}^{m-1},N_{m-1}}] dt 
	+ \sigma [{\overline{Y}}^{i,m}_{\eta(t)},\mu_{\eta (t)}^{\overline{Y}^{m-1},N_{m-1}}] dW^{i,m}_t, 	
\end{equation}
and call it an \emph{iterative particle system}. As above, we require that $W^{i,m},$ $1 \leq i \leq N_m,$ $m \in \bN$, and ${\overline{Y}}^{i,m}_0,$ $1 \leq i \leq N_m,$ $m \in \bN$, are independent. By this construction, the particles $({\overline{Y}}^{i,m}_t)_{i}$ are independent upon conditioning on $\sigma \Big( { {\{ {\overline{Y}}^{i,m-1}_t \}_{
1 \leq i \leq N_{m-1} }}}: t\in[0,T] \Big)$. 
The error analysis of \eqref{eq:xminteract2} is presented in Theorem \eqref{lm:mseboundfinalinteractparticle} and \eqref{thm:c2nonintmc}. From there one can deduce that optimal computational cost is achieved when $\{N_m\}_m$ is increasing and the computational complexity of computing expectations with \eqref{eq:xminteract2} is of the same order as the original particle system, i.e. $\eps^{-5}$. 
 
\subsection{Main result of the iterative MLMC algorithm} 
To reduce the computational cost, we combine the MLMC method with Picard iteration \eqref{eq:defxtminteract}. Fix $m$ and $L$.  Let  $\Pi^\ell=\{0=t_{0}^{\ell},\ldots,t^{\ell}_k,\ldots, T=t_{2^{\ell}}^\ell \} $, $\ell=0,\ldots,L$, be a family of time grids such that $t^{\ell}_k - t^{\ell}_{k-1}=h_{\ell} =  T2^{-\ell}$. To simulate \eqref{eq:xminteract} at Picard step $m$ and for all discretisation levels $\ell$ we need to have an approximation of the relevant expectations with respect to the law of the process at the previous Picard step $m-1$ and the time grid $\Pi^{L}$, i.e.   
\[
\begin{split}
	 &\left( \bE[b(x,\overline {X}^{m-1}_0)],\ldots, \bE[b(x,\overline{X}^{m-1}_{t_k^{L}})],\ldots,  \bE[b(x,\overline{X}^{m-1}_{T})]\right)\,,\\
	 &\left( \bE[\sigma(x,\overline {X}^{m-1}_0)],\ldots, \bE[\sigma(x,\overline{X}^{m-1}_{t_k^{L}})],\ldots,  \bE[\sigma(x,\overline{X}^{m-1}_{T})]\right)\,.
\end{split}
\]  
By approximating these expectations with the MLMC (signed) measure $\mathcal{M}^{(m-1)}$ (see Section \ref{sec MLMC} for its exact definition), we arrive at the \textit{iterative MLMC particle} method defined as
\begin{equation} \label{eq:contiousYinteract}
 dY^{i,m,\ell}_t = \langle \mathcal{M}^{(m-1)}_{\eta_{\ell} (t)} ,  b( Y^{i,m,\ell}_{\eta_{\ell} (t)}, \cdot)  \rangle \, dt + \langle \mathcal{M}^{(m-1)}_{\eta_{\ell} (t)} ,  \sigma( Y^{i,m,\ell}_{\eta_{\ell} (t)}, \cdot) \rangle \, dW^{i,m}_t\,, 
\end{equation}
where $Y^{i,0, \ell}= X_0$. Under the assumptions listed in Section \ref{sec MLMCI}, the main result of this paper gives precise error bounds for \eqref{eq:contiousYinteract}.
\begin{theorem}\label{lm:mseboundfinalinteract}
Assume  \hkreg \, and \hLaw .  Fix $M>0$ and let $P \in C^2_b (\bR^d)$.  Define
$
MSE_t^{(M)}(P) \defeq \bE[( \langle \mathcal{M}^{(M)}_{t}, P \rangle - \bE[P(X_{t})] )^2].
$
Then there exists a constant $c>0$ (independent of the choices of $M$, $L$ and $\{N_{m,\ell}\}_{m,\ell}$) such that for every $t \in [0,T]$,
\[
	MSE_{\eta_L(t)}^{(M)} (P)\leq c\bigg\{h_L^2+\sum_{m=1}^{M}\frac{c^{M-m}}{(M-m)!}\cdot\sum_{\ell=0}^L\frac{h_\ell}{N_{m,\ell}} +\frac{c^{M-1}}{M!}\bigg\}.
\]
\end{theorem}
The proof can be found in Section \ref{sec mseboundfinalinteract}. The first term in the above error comes from the analysis of weak convergence for the Euler scheme. The second contains the usual MLMC variance and shows that computational effort should be increasing with with iteration $m$ (rather than equally distributed across iterations). Finally the last term is an extra error due to iterations. Using this result, we prove in Theorem \ref{thm:c2nonint} that the overall complexity of the algorithm is of order $\eps^{-4}|\log\eps|^{3}$ (i.e. one order of magnitude better than the direct approach). 
We remark that the MLMC measure acts on functionals that depend on spatial variables. We work with uniform norms as in  \cite{heinrich2001multilevel,giles2015multilevel}, but also introduce suitable weighted norms, which seems new in MLMC literature.

We remark that, the analysis of stochastic particles systems is of independent interest, as it is used as models in molecular dynamics; physical particles in fluid dynamics \cite{pope2000turbulent}; behaviour of interacting agents in economics or social networks \cite{carmona2013control} or interacting neurons in biology \cite{delarue2015particle}. It is also used in  modelling networks of neurons (see \cite{delarue2015global}) and
modelling altruism (see \cite{hutzenthaler2016full}).

\subsection{Convention of notations}
 We use $\norm{A}$ to denote the Hilbert-Schmidt norm while $|{\bf v}|$ is used to denote the Euclidean norm.  For any stochastic process $R= \{R_t \}_{t \in I}$, the law of $R_t$ at any time point $t \in I$ is denoted by $\mu^R_t$. 
$\cP_2( E)$ denotes the set of square-integrable probability measures on any Polish space $E$. On the other hand, $\cP^{s}_2( E)$ denotes, on any Polish space $E$, the set of random signed measures that are square-integrable almost surely. 
 
 Moreover, we denote by $C^{0,2}_{b,p}(\bR^m \times \bR^n, \bR)$ the set of functions $P$ from $\bR^m \times \bR^n $ to $\bR$ that are continuously twice-differentiable in the second argument, for which there exists a constant $L$ such that for each $x \in \bR^m$, $y \in \bR^n$, $i,j \in \{1, \ldots, n\}$,
$$ |\partial_{y_i}P(x,y)| \leq L(1 + |y|^p), \quad \quad |\partial^2_{y_i,y_j}P(x,y)| \leq L(1 + |y|^p), $$
where $\partial_{y_i}$ and $\partial^2_{y_i,y_j}$ denote respectively the first and second order partial derivatives w.r.t. the second argument. Finally, we denote by $C^{p,q}_{b,b}(\bR^m \times \bR^n, \bR)$ the set of functions from $\bR^m \times \bR^n $ to $\bR$ that are continuously $p$ times differentiable in the first argument and continuously $q$ times differentiable in the second argument such that the partial derivatives (up to the respective orders, excluding the ``zeroth'' order derivative) are bounded.

\section{The iterative MLMC algorithm} \label{sec MLMCI}
\subsection{Main assumptions on the McKean-Vlasov SDE} \label{sec:abstract}
Here we state the assumptions needed for the analysis of equation \eqref{eq:general1}. 
\begin{assumption} \label{as 1}
 $  $
\begin{itemize}
\item[\hkreg] The kernels $b$ and $\sigma$  belong to the sets  $C^{2,1}_{b,b}(\R^d\times\R^{d}, \R^d) \cap C^{0,2}_{b,p}(\R^d\times\R^{d}, \R^d)$ and $C^{2,1}_{b,b}(\R^d\times\R^{d}, \R^{d\otimes r})  \cap C^{0,2}_{b,p}(\R^d\times\R^{d}, \R^{d\otimes r})$ respectively.
\item[\hLaw]  The initial law $\mu_0 := \mu^X_0$ satisfies the following condition:
 for any $p\geq 1$, $\mu_0\in L^p(\Omega;\R^d)$, i.e. $$\int_{\R^d}|x|^p\mu_0(dx)<\infty.$$
\end{itemize}
\end{assumption}

Note that if \hkreg \,  holds, then 
\begin{itemize}
\item[\hLip] the kernels $b$ and $\sigma$ are globally Lipschitz,  i.e. 
for all $x_1,x_2, y_1, y_2\in\R^d$, there exists a constant $L$ such that
$$ | b(x_1,y_1) -b(x_2,y_2) |  +  \| \sigma(x_1,y_1) - \sigma(x_2,y_2) \|  \leq L ( |x_1-x_2| + |y_1-y_2|). $$
\end{itemize}
If \hkreg \, and \hLaw \, hold, then a weak solution to \eqref{eq:general1} exists and pathwise uniqueness holds (see \cite{sznitman1991topics}). In other words $\{ X_t \}_{t \geq 0}$  induces a unique probability measure on  $C([0,T],\R^d)$ . Furthermore it has a property that 
\begin{align}\label{eq:regularityofXinteract}
\sup_{0\leq t\leq T}\bE|X_t|^p <\infty.
\end{align}
The additional smoothness stipulated in \hkreg \, is needed in the analysis of weak approximation errors.  

\subsection{Direct application of MLMC to interacting diffusions}\label{sec issues}
There are  two issues pertaining to the direct application of  MLMC methodology to \eqref{eq:sEuler2}: 
i) the telescopic property needed for MLMC identity \cite{giles2008multilevel} does not hold in general; ii) a small number of simulations (particles) on fine time steps (a reason for the improved computational cost in MLMC setting) would lead to a poor approximation of the measure, leading to  a high bias.
 To show that telescopic sum does not hold in general, consider a collection of discretisations of $[0,T]$ with different resolutions. To this end, we fix $L\in \NN$. Then $Y^{i,\ell,N_{\ell}}_{T}$, $\ell=1,\ldots,L$, denotes for each $i$ a particle corresponding to \eqref{eq:sEuler2} with time-step $h_\ell$, where $N_{\ell}$ is the total number of particles. Let $P: \bR^d \to \bR$ be any Borel-measurable function.
With a direct application of MLMC in time for \eqref{eq:sEuler2}, we replace the standard Monte-Carlo estimator on the left-hand side by an MLMC estimator on the right-hand side as follows.
\begin{eqnarray}
    && \frac{1}{N_{L}} \sum_{i=1}^{N_{L}}    P( Y^{i, L, N_{L}}_t)  \nonumber \\
    & \approx & \frac{1}{N_0}  \sum_{i=1}^{N_{0}}  P( Y^{i, 0, N_{0}}_t) + \sum_{\ell=0}^{L} \frac{1}{N_{\ell}}  \sum_{i=1}^{N_{\ell}} \bigg[ P( Y^{i, \ell, N_{\ell}}_t)  -  P( Y^{i, \ell-1, N_{\ell}}_t)  \bigg]. \label{eq: naive application MLMC} 
\end{eqnarray} 
 However, we observe that such a direct application is not possible, since, in general,   
\begin{equation*}
\bE \bigg[  P( Y^{1, \ell, N_{\ell}}_t) \bigg] \neq \bE \bigg[ 
 P( Y^{1, \ell, N_{\ell+1}}_t) \bigg], 
\end{equation*}
which means that we do not have equality in expectation on both sides of \eqref{eq: naive application MLMC}. On the contrary, if we required the number of particles for all the levels to be the same, then the telescopic sum would hold, but clearly, there would be no computational gain from doing MLMC. We are aware of two articles that tackle the aforementioned issue. The case of linear coefficients is treated in \cite{ricketson2015multilevel}, in which particles from all levels are used to approximate the mean field at the final (most accurate) approximation level. It is not clear how this approach could be extended to general McKean-Vlasov equations. A numerical study of a ``multi-cloud" approach is presented in \cite{haji2016multilevel}. The algorithm resembles the MLMC approach to the nested simulation problem in \cite{ali2012pedestrian,giles2015multilevel,bujok2013multilevel,lemaire2017multilevel}. Their approach is very natural, but because particles within each cloud are not independent, one faces similar challenges as with the classical particle system.

\subsection{Construction of the iterative MLMC algorithm}\label{sec MLMC}

We approximate each of the expectations by the MLMC method, but only have access to samples at grid points ${\Pi}^{\ell}$ that correspond to $(Y^{i,m-1,\ell})_{i,\ell}$. Consequently, for $\ell<\ell'$, the empirical measure $ \frac{1}{N} \sum_{i=1}^N \delta_{Y^{i,m-1,\ell}_{t}}$ is only defined at every timepoint in $\Pi^{\ell}$,  but not $\Pi^{\ell'}$ and one cannot build MLMC telescopic sum across all discretisation levels. For that reason (as in original development of MLMC by Heinrich \cite{heinrich2001multilevel}), we introduce a linear-interpolated measure (in time) ${\widetilde{\mu}}^{Y^{m-1,{\ell}},N}_t$ given by
\begin{align} \label{eq:intmeasureintro}
 {\widetilde{\mu}}^{Y^{m-1,{\ell}},N}_t \defeq
\begin{cases}
 \frac{1}{N} \sum_{i=1}^N \delta_{Y^{i,m-1,\ell}_t}   &, t \in \Pi^{\ell}, \\
  {} & {} \\
	\bigg[ \frac{ t- \eta_{\ell} (t)}{h_{\ell}} \bigg]   {\widetilde{\mu}}^{Y^{m-1,{\ell}},N}_{\eta_{\ell} (t) + h_{\ell}}   + \bigg[ 1- \frac{ t- \eta_{\ell} (t)}{h_{\ell}} \bigg] {\widetilde{\mu}}^{Y^{m-1,{\ell}},N}_{\eta_{\ell} (t) }  & ,t \notin \Pi^{\ell}\,,
 \end{cases}
\end{align}
where $\eta_\ell{(t)}\defeq t_k^\ell,\quad\text{if}~t\in[t_k^{\ell},t_{k+1}^\ell)$.
For any continuous function $P:\bR^d\times\bR^d \rightarrow \bR$ and any $ x\in \bR^d$, we define the MLMC signed measure $\mathcal{M}^{(m-1)}_t$ by
\begin{gather} \label{eq:defMLMCop}
 \langle \mathcal{M}^{m-1}_t,  P(x,\cdot) \rangle := \ \langle \sum_{\ell=0}^L ( {\widetilde{\mu}}^{Y^{m-1,{\ell}},N_{m-1,\ell}}_t -{\widetilde{\mu}}^{Y^{m-1,{\ell-1}},N_{m-1,\ell}}_t  ) , P(x,\cdot) \rangle\,,
 \end{gather}
 where $ {\widetilde{\mu}}^{Y^{m-1,{-1}},N_{m,0}}_t :=0$. We interpret the MLMC operator in a componentwise sense.  We then define the particle system $\{ Y^{i,m,\ell} \}$ as in \eqref{eq:contiousYinteract}.  As usual for MLMC estimators, at each level $\ell$, we use the same Brownian motion to simulate particle systems  $(Y^{i,m,\ell},Y^{i,m,\ell-1})_{i}$ to ensure 
that the variance of the overall estimator is reduced. As for the iterative particle system, we require that $W^{i,m},$ $1 \leq i \leq N_{m, \ell},$ $m \in \bN$, and ${Y}^{i,m,\ell}_0,$ $1 \leq i \leq N_{m, \ell},$ $1 \leq \ell \leq L$, $m \in \bN$, are independent.

\section{Abstract framework for MLMC analysis}

To streamline the analysis of the iterated MLMC estimator, we introduce an abstract framework corresponding to one iteration. This simplifies the notation and also may be useful for future developments of MLMC algorithms. 

Let  $\overline{b}: \mathbb{R}^d \times \cP^s_2(\mathbb{R}^{d}) \rightarrow \mathbb{R}^d$ and $\overline{\sigma}: \mathbb{R}^d \times \cP^s_2(\mathbb{R}^{d}) \rightarrow \mathbb{R}^{d\otimes r}$ be measurable functions. Also, $\cV \in  \cP^s_2(C([0,T], \bR^d))$ is fixed (the precise conditions that we impose on $\overline{b}$, $\overline{\sigma}$ and
$\mathcal{V}$ will be presented in Section \ref{sec:aa}). 
We consider SDEs with \emph{random} coefficients of the form
\begin{equation} \label{eq:SDErandom}
	dU_t = \overline{b}(U_t, \cV_{t} )dt +\overline{\sigma}(U_t, \cV_{t} )dW_t, \quad \mu^U_0=\mu^X_0.
\end{equation}
The solution of this SDE is well-defined under the assumptions in Section \ref{sec:aa}, by \cite{krylov2002introduction}. For $\ell=1,\ldots,L$, the corresponding Euler approximation of \eqref{eq:SDErandom} at level $\ell$ is given by  
\begin{equation} \label{eq:defprocessZtinteract}
dZ^{\ell}_{t} =   \overline{b}(Z_{\eta_{\ell} (t)}^{\ell},\cV_{\eta_{\ell} (t)} ) dt + \overline{\sigma}(Z_{\eta_{\ell} (t)}^{\ell},\cV_{\eta_{\ell} (t)}) dW_t, \quad \mu^{Z^{\ell}}_0 = \mu^X_0.
\end{equation}
We require that $\cV$ does not depend on $\ell$  and that $(W_t)_{t \in [0,T]}$ is independent of $\cV$. Subsequently, we define a particle system $\{Z^{i, \ell}\}$ as follows,
\begin{equation} \label{eq:defprocessZtcoupling}
dZ^{i,\ell}_{t} =   \overline{b}(Z_{\eta_{\ell} (t)}^{i,\ell},\cV_{\eta_{\ell} (t)} ) dt + \overline{\sigma}(Z_{\eta_{\ell}(t)}^{i,\ell},\cV_{\eta_{\ell} (t)}) dW^{i}_t, \quad \quad \mu^{Z^{i,\ell}}_0 = \mu^X_0\,.
\end{equation}

\subsection{Analysis of the abstract framework} \label{sec:aa}

Using the notation defined in the previous section, we formulate the conditions needed to study the convergence of the iterated particle system. Recall that $\cV \in  \cP^s_2(C([0,T], \bR^d))$ is given and we consider equations \eqref{eq:defprocessZtinteract} and \eqref{eq:defprocessZtcoupling}. We assume the following.

\begin{assumption}
\hfill
\begin{itemize}

\item[\hvb] The random measure $\cV$ is independent of $W^{i}$ and $Z^{i, \ell}_0$. For each $p \geq 1$, 
\[ \quad \quad 
	 \sup_{0\leq s\leq T}\bE \bigg| \int_\Rd |y|^p\cV_{s}(dy)  \bigg| < \infty.
\]

\item[\hvreg] There exists a constant $c$ such that
\[
    	\sup_{x\in\R^d}\sup_{0\leq s\leq t\leq T}\bE \big[ \big| \overline{b}(x,\cV_t)-\overline{b}(x,\cV_s)\big|^2 + \big\| \overline{\sigma}(x,\cV_t)-\overline{\sigma}(x,\cV_s) \big\|^2 \big]\leq c (t-s).
\]

\item[\hvLip] There exists a constant $c$ such that for each $t\in[0,T]$ and $x,y\in\R^d$, 
\begin{gather}
	|\overline{b}( x, \cV_t) - \overline{b}(y,  \cV_t)| +\norm{\overline{\sigma}( x, \cV_t ) - \overline{\sigma}(y, \cV_t )}\leq c|x-y|\,
\end{gather}
\begin{gather}
	|\overline{b}(x, \cV_t)| + \norm{\overline{\sigma}(x, \cV_t)} \leq c \bigg( 1+|x|+ \bigg| \int_\Rd|y| \, \cV_t(dy) \bigg| \bigg)\,.
\end{gather}

\end{itemize}
\end{assumption}

\paragraph{Analysis of conditional MLMC variance}
For the rest of this section, we denote by $c$  a generic constant that depends on $T$, but not on $\ell$ or $N_{\ell}$.  We first consider the integrability of process \eqref{eq:defprocessZtinteract}.

\begin{lemma}\label{lm:intergrabilitySDErandomcoeffinteract}
Let $Z^\ell$ be defined as in \eqref{eq:defprocessZtinteract}. Assume \hvLip \, and \hLaw. Then for any $p\geq 2$ and $\ell\geq 0$, there exists a constant $c$ such that 
\[
	\bE \bigg[\sup_{t\in[0,T]}|Z_t^\ell|^p \bigg]\leq c\ \bigg(1+\bE \bigg[\int_0^T \bigg| \int_\Rd|y|^p\cV_{\eta_{\ell}(s)}(dy)  \bigg| \, ds \bigg] \bigg).
\]
\end{lemma}
The proof is elementary and can be found in the Appendix \ref{sec intergrabilitySDErandomcoeffinteract}. The following two lemmas focus on the regularity of $Z^{\ell}_t$ in time and its strong convergence property.  The first lemma bounds the difference in $Z^{\ell}_t$ over two time points, at a fixed level $\ell$. The second lemma bounds the difference in $Z^{\ell}_t$ over adjacent levels, at a fixed time $t$.  Their proofs follow from standard estimates in the theory of SDE and are therefore omitted. 

\begin{lemma}[Regularity  of $Z^{\ell}_t$]\label{lm:regularityYinteract}
Let $Z^\ell$ be defined as in \eqref{eq:defprocessZtinteract}. Assume \hvLip $\,$ and \hvb. Then, for $p\geq 1$, $0 \leq u \leq s \leq T$, 
\[
	\bigg(\bE[|Z_s^{\ell}-Z_{u}^{\ell}|^p]\bigg)^{\frac{1}{p}}\leq c (s-u)^{\frac{1}{2}}.
\]
\end{lemma}

\begin{lemma}[Strong convergence of $Z^{\ell}_t$]\label{lm:variance1nonintinteract}
Assume \hvLip, \hvb \, and \hvreg. Then for any $\ell \in \{ 1,2,\ldots, L \}$, there exists a constant $c>0$ such that

\[
	 \bE \bigg[ \sup_{0\leq t\leq T}|Z_{t}^{\ell}-Z_{t}^{\ell-1}|^2 \bigg]\leq c h_\ell.
\]
\end{lemma}

We define the interpolated empirical measures  $ \widetilde \mu_t^{Z^{\ell},N}$ exactly as in \eqref{eq:intmeasureintro} 
and the corresponding MLMC operator $\mathcal{M}_t$  (corresponding to \eqref{eq:defMLMCop}, but for one Picard iteration) as
\[
 \langle \mathcal{M}_t , P(x,\cdot) \rangle  = \left \langle \sum_{\ell=0}^L \Big( \widetilde \mu_t^{Z^{\ell},N_{\ell}} -\widetilde \mu_t^{Z^{\ell-1},N_{\ell}} \Big) , P(x,\cdot) \right \rangle, \quad \quad \widetilde \mu_t^{Z^{-1},N_{0}} :=0.
 \]
We also define $\sigma$-algebra $\F^\cV_t=\{\sigma(\cV_{s})_{0\leq s \leq t} \}$. Since samples $\{Z_{\eta_L(t)}^{i,\ell}\}$,  $i=1,\ldots,N_{\ell}$,  $\ell =0, \ldots, L$, conditioned on $\F^\cV_T$ are independent, we can bound the conditional MLMC variance as follows.

\begin{lemma}\label{lm:varianceinterpvsnoninterpinteract}
Assume \hvLip, \hvb \, and \hvreg \,  hold. Let $\mu \in \cP_2\big( C([0,T],\bR^d) \big)$. Then for any Lipschitz function $P:\bR^d \times \bR^d \to \bR$, there exists a constant $c$ such that
\begin{align}\label{eq:interpstrongerrorinteract}
	\sup_{0\leq t\leq T} \int_{\bR^d} \bE \bigg[ \text{\emph{Var}} \bigg( \langle \mathcal{M}_{\eta_L(t)} , P(x,\cdot) \rangle    \bigg|\F^\cV_T \bigg) \bigg]  \, \mu_t (dx) \leq c\sum_{\ell=0}^L \frac{ h_\ell}{N_{\ell}}.
\end{align}
\end{lemma}
\begin{proof}
The independence condition in  \hvb $\,$ implies that
\begin{eqnarray}  && \bE \bigg[ \text{Var}  \bigg( \langle \mathcal{M}_{\eta_L(t)} , P(x,\cdot) \rangle  \bigg|\F^\cV_T \bigg) \bigg] \nonumber \\
& =  & \sum_{i=1}^{N_{0}} \frac{1}{N^2_{0}} \bE \bigg[ \text{Var} \Big[ {P}^{i, 0}_{\eta_L(t)}   \Big| \F^\cV_T \Big] \bigg]+ \sum_{\ell=1}^L \sum_{i=1}^{N_{\ell}} \frac{1}{N^2_{\ell }} \bE \bigg[ \text{Var} \Big[ {P}^{i, \ell}_{\eta_L(t)} -  {P}^{i, \ell-1}_{\eta_L(t)} \Big| \F^\cV_T \Big] \bigg], \nonumber   \end{eqnarray}
where 
\begin{equation} \label{lm:definitionP}
{P}^{i, \ell}_{\eta_L(t)} : = (1-\lambda^{\ell}_t)P(x,Z_{ \eta_{\ell} (\eta_L (t))}^{i,\ell})+\lambda^{\ell}_t P(x,Z_{{ \eta_{\ell} (\eta_L (t))+h_\ell}}^{i,\ell}),
\end{equation}
$\lambda^{\ell}_t = \frac{ \eta_L (t)- \eta_{\ell} ( \eta_L (t))}{h_{\ell}}\in[0,1]$. 
Using the fact that $\bE[\text{Var}(X|\mathcal{G})]\leq \text{Var}(X) \leq \bE[X^2]$, we obtain the bound
\begin{gather} \bE \bigg[ \text{Var} \bigg( \langle \mathcal{M}_{\eta_L(t)} , P(x,\cdot) \rangle \bigg|\F^\cV_T \bigg) \bigg] \leq  \sum_{i=1}^{N_{0}} \frac{1}{N^2_{0}} \bE \bigg| {P}^{i, 0}_{\eta_L(t)} \bigg|^2 + \sum_{\ell=1}^L \sum_{i=1}^{N_{\ell}} \frac{1}{N^2_{\ell }} \bE \bigg| {P}^{i, \ell}_{\eta_L(t)} -  {P}^{i, \ell-1}_{\eta_L(t)} \bigg|^2. \nonumber \end{gather}
Since $P$ is Lipschitz, it has linear growth. By Lemma \ref{lm:intergrabilitySDErandomcoeffinteract}, it follows that
\begin{gather}
    \bE \bigg| {P}^{i, 0}_{\eta_L(t)} \bigg|^2 \leq c \sup_{0 \leq t \leq T} \int_{\bR^d} \bigg( x^2 + \bE \big| Z^{i,0}_{\eta_0 (\eta_L(t))} \big|^2 +  \bE \big| Z^{i,0}_{\eta_0 (\eta_L(t)) +h_0} \big|^2 \, \bigg) \mu_t (dx) < +\infty. \nonumber 
\end{gather}
Next, we consider levels $\ell \in \{1, \ldots, L\}$. Recall from (\ref{lm:definitionP}) that
\begin{align*}
	{P}^{i, \ell}_{\eta_L(t)} &= (1-\lambda^{\ell}_t)P(x,Z_{ \eta_{\ell} (\eta_L (t))}^{i,\ell})+\lambda^{\ell}_t P(x,Z_{{ \eta_{\ell} (\eta_L (t))+h_\ell}}^{i,\ell}),\\
	{P}^{i, \ell-1}_{\eta_L(t)} &= (1-\lambda^{\ell-1}_t)P(x,Z_{\eta_{\ell-1}(\eta_L(t))}^{i,\ell-1})+\lambda^{\ell-1}_t P(x,Z_{\eta_{\ell-1}(\eta_L(t))+h_{\ell-1}}^{i,\ell-1}).
\end{align*}
 We decompose the error as follows. \small
{\begin{eqnarray}
&& |{P}^{i, \ell}_{\eta_L(t)}-{P}^{i, \ell-1}_{\eta_L(t)}| \nonumber \\
&\leq & (1-\lambda^{\ell-1}_t)\cdot\bigg|P(x,Z_{{\eta_\ell(\eta_L(t))}}^{i,\ell})\pm P(x,Z^{i,\ell-1}_{{\eta_\ell(\eta_L(t))}})-P(x,Z_{\eta_{\ell-1}(\eta_L(t))}^{i,\ell-1})\bigg| \nonumber\\
&&+\, \lambda^{\ell-1}_t\cdot\bigg|P(x,Z_{\eta_\ell(\eta_L(t))+h_\ell}^{i,\ell})\pm P(x,Z^{i,\ell-1}_{\eta_\ell(\eta_L(t))+h_\ell}) - P(x,Z_{\eta_{\ell-1}(\eta_L(t))+h_{\ell-1}}^{i,\ell-1})\bigg| \nonumber \\
&& +\, |\lambda^{\ell}_t-\lambda^{\ell-1}_t|\cdot\bigg|P(x,Z^{i,\ell}_{\eta_\ell(\eta_L(t))+h_\ell})- P(x,Z^{i,\ell}_{\eta_\ell(\eta_L(t))})\bigg|. \nonumber
\end{eqnarray}}
\normalsize{By Lemma \ref{lm:variance1nonintinteract},} \small 
{\begin{eqnarray}\label{eq:lm:varianceinterpvsinterp1interact}
\bE|P(x,Z_{\eta_\ell(\eta_L(t))}^{i,\ell})-P(x,Z^{i,\ell-1}_{\eta_\ell(\eta_L(t))})|^2 &\leq & c h_{\ell},\\
\bE|P(x,Z_{\eta_\ell(\eta_L(t))+h_\ell}^{i,\ell})-P(x,Z^{i,\ell-1}_{\eta_\ell(\eta_L(t))+h_\ell})|^2 &\leq & c h_{\ell}. \label{eq:lm:varianceinterpvsinterp1interact2}
\end{eqnarray}}
\normalsize{Also, by Lemma \ref{lm:regularityYinteract},} \small 
{\begin{equation}
    \bE|P(x,Z^{i,\ell-1}_{\eta_\ell(\eta_L(t))})-P(x,Z_{\eta_{\ell-1}(\eta_L(t))}^{i,\ell-1})|^2  \leq  c (\eta_\ell(\eta_L(t))-\eta_{\ell-1}(\eta_L(t)))\leq c h_\ell,   \label{eq:lm:varianceinterpvsinterp2interact} 
\end{equation}
\begin{equation}
    \bE|P(x,Z^{i,\ell-1}_{\eta_\ell(\eta_L(t))+h_\ell}) - P(x,Z_{\eta_{\ell-1}(\eta_L(t))+h_{\ell-1}}^{i,\ell-1})|^2  \leq  c h_{\ell},  \label{eq:lm:varianceinterpvsinterp3interact} 
\end{equation}}
\normalsize{and} \small{
\begin{equation}
    \bE|P(x,Z^{i,\ell}_{\eta_{\ell}(\eta_L(t))+h_{\ell}})- P(x,Z^{i,\ell}_{\eta_{\ell}(\eta_L(t))})|^2  \leq  c h_{\ell}. \label{eq:lm:varianceinterpvsinterp4interact} 
\end{equation}}
\normalsize{We obtain \eqref{eq:interpstrongerrorinteract} by combining \eqref{eq:lm:varianceinterpvsinterp1interact}, \eqref{eq:lm:varianceinterpvsinterp1interact2}, \eqref{eq:lm:varianceinterpvsinterp2interact}, \eqref{eq:lm:varianceinterpvsinterp3interact} and \eqref{eq:lm:varianceinterpvsinterp4interact}. 
Since $t$ and $x$ are arbitrary, the proof is complete.}

\end{proof}

\subsection{Weak error analysis} 
We begin this subsection by defining  $\cX^{s,x}$ 
as
$$ \cX^{s,x}_t = x + \int_s^t b[ \cX^{s,x}_u, \mu^X_u] \,du +\int_s^t \sigma[ \cX^{s,x}_u, \mu^X_u] \,dW_u.  $$
 For  $P \in C^{0,2}_{b,b} (\bR^d \times \bR^d, \bR)$ and $t \in [0,T],$ we consider the function
\begin{align}\label{eq:defvinteract}
v_y(s,x):=\E[P(y,\cX_t^{s,x})], \quad y\in\R^d~\text{and}~(s,x)\in[0,t]\times \RR^d.
\end{align}
We aim to show that $v_y(s,x)\in C^{1,2}$. The first step is the lemma below.    
\begin{lemma}\label{lm:regularityBinteract}
Assume  \hLaw \, and \hkreg.   Then 
\[
	b[\cdot,\mu_{\cdot}^X]\in C^{2,1}_{b,b}(\R^d \times [0,T],\R^d)~\text{and}~\sigma[\cdot,\mu_{\cdot}^X]\in C^{2,1}_{b,b}(\R^d \times [0,T],\R^{d\otimes r}).
\]
\end{lemma}
\begin{proof}
 For any $x\in\R^d$, $s\in[0,T]$ and $t\in[s,T]$, we apply It\^{o}'s formula to each coordinate $k\in\{1,\ldots,d\}$ of $b$ to get
\begin{equation}\label{eq:lmregularityB1}
\begin{split}
	b_k(x,X_t)&=b_k(x,X_s) + \int_s^t\smfir\smfi {\partial_{y_j} b_k}(x,X_u)\sigma_{ji}[X_u,\mu_u^X]dW_u^i\\&\ + \int_s^t\smfir{\partial_{y_j}  b_k} (x,X_u) b_j[X_u, \mu_u^X ]du+\frac{1}{2}\int_s^t\smsec{\partial^2_{y_i,y_j}  b_k}(x,X_u)a_{ij}[X_u,\mu_u^X]du,
\end{split}
\end{equation}
where $a[x,\mu]=\sigma[x,\mu]\sigma[x,\mu]\tr$ and $\partial_{y_i} b_k, \partial^2_{y_i,y_j}  b_k$ indicate the derivatives w.r.t. the  the second argument. Assumptions \hkreg, \, \hLip, \, \hLaw  \, and \eqref{eq:regularityofXinteract} imply that the above stochastic integral is a martingale. By the fundamental theorem of calculus,
\begin{gather}
		\partial_t\bE[b_{k}(x, X_t)] = \bE\bigg[\smfir \partial_{y_j}  b_k(x, X_t) b_j[x,\mu_t^X] +\frac{1}{2}\smsec \partial^2_{y_i,y_j}  b_k(x,X_t)a_{ij}[x,\mu_t^X]\bigg].
\end{gather}

By  \hkreg, $\partial_{y_j}  b_k$ and $\partial^2_{y_i,y_j}  b_k$ are bounded. Moreover, by \hLip, we know that $b$ and $a$ are respectively of linear and quadratic growth in $x$. Therefore, by \eqref{eq:regularityofXinteract}, we conclude that $\partial_t b_k[x,\mu_t^X] $ is bounded.
To conclude, we can apply the same argument to $\sigma[\cdot,\mu_\cdot^X]$.\end{proof}

\begin{lemma} \label{eq:regularityofderivativeofvy}${}$ Assume \hkreg \, and \hLaw. Then for any $(s,x) \in [0,t] \times \bR^d$,  $(i,j)\in\{1,\ldots,d\}^2$ and $P \in C^{0,2}_{b,b}(\bR^d \times \bR^d, \bR)$,
\[ \tag{\hvdiffs}
      \sup_{y\in\R^d} ( \|\partial_{x_i} v_y (s,x) \|_{\infty}   + \|\partial^2_{x_i, x_j} v_y  (s,x) \|_{\infty} \leq L. 
\]  

\end{lemma}
\begin{proof}
We only provide a sketch as the argument is standard. By the fact that  the first-order spatial derivatives of $b[\cdot,\mu_{\cdot}^X]$ and $\sigma[\cdot,\mu_{\cdot}^X]$ are bounded, it is straightforward to deduce that
\begin{equation} \label{eq:boundedfirstorder} 
\sup_{x \in \mathbb{R}^d} \sup_{s \in [0,t]} \bE \bigg[  \bigg| \partial_{x_i} (X^{s,x}_t)^{(j)} \bigg|^2 \bigg] <  \infty.
\end{equation}
Theorem 5.5.5 in \cite{friedman2006stochastic} establishes that
\begin{equation} \label{eq:vyfirstderivative} \partial_{x_i} v_y  (s,x) = \sum_{j=1}^d \bE \bigg[ \partial_{y_j} P(y, X^{s,x}_t) \partial_{x_i} (X^{s,x}_t)^{(j)} \bigg].\end{equation}
By \eqref{eq:boundedfirstorder}, it is clear that the assertion for the first order derivatives in (\hvdiffs)~ holds if $P \in C^{0,2}_{b,b}(\bR^d \times \bR^d, \bR)$. Similarly, we can prove the assertion for the second order derivatives in the same way.
\end{proof}

By the Feynman-Kac theorem (\cite{MR601776}), it can be shown that  $v_y(\cdot,\cdot)$ satisfies the following Cauchy problem, 
\begin{equation}\label{eq:proofalphafinalnonintinteract}
 \left\{
\begin{array}{rl}
        \partial_s v_y (s,x)  + \displaystyle\frac{1}{2}\smsec  \Big( \sigma[x, \mu^X_s] \sigma[x, \mu^X_s]^T \Big)_{ij}  \partial^2_{x_i,x_j} v_y (s,x) &\\
       +\displaystyle\sum_{j=1}^d \Big( b[x, \mu^X_s] \Big)_j  \partial_{x_j} v_y  (s,x) &=0, \quad (s,x)\in[0,t]\times \RR^d,\\
        v_y (t,x) &= P(y,x). \\
\end{array} 
\right.
\end{equation}

The following theorem reveals the order of weak convergence of \eqref{eq:defprocessZtinteract} to \eqref{eq:general1}. We denote by ${\mu}^{Z^{{\ell}}|\F_T^\cV}_{t}$ the regular conditional probability measure of $Z^{{\ell}}_t$ given $F_T^\cV$. (See Theorem 7.1 in \cite{parthasarathy1967probability} for details.) The existence of regular conditional probability measure follows from the fact that we work on a Polish space with the Borel $\sigma-$algebra. 


\begin{theorem}\label{lm:conditionalwknonintmaininteract}
Let $P \in C^{0,2}_{b,b}( \bR^d \times \bR^d, \bR)$ be a Lipschitz continuous function. \footnote{ Note that the regularity of $P$ can be relaxed to $C^{0,2}_{b,p}( \bR^d \times \bR^d, \bR)$. We prove the result in a slightly stronger assumption for the sake of simplicity.}  Assume that \hkreg, \hLaw, \hvb \,  and \hvLip $\,$ hold. 
Then there exists a constant $c$ $($independent of the choices of $L$ and $N_1, \ldots, N_{L})$ such that for each $t\in[0,T]$, $\ell \in \{0, \ldots, L \}$ and $x\in\R^d$,
\begin{align*}
\sup_{0\leq s\leq t} & |\bE[P(x,Z_{s}^{\ell})]-\bE[P(x,X_{s})]| \\
&\leq c\bigg(h_{\ell}+\int_{0}^{t} \bE\bigg[\int_{\R^d} \Big|\overline{b}(x,\cV_{\eta_{\ell}(s)})- \bE[b(x,X_{\eta_{\ell}(s)})] \Big| \, {\mu}^{Z^{{\ell}}|\F_T^\cV}_{\eta_{\ell}(s)} (dx)\bigg]ds\\
&\hspace{1cm}+\int_{0}^{t} \bE\bigg[\int_{\R^d} \Big\|{\overline{\sigma}(x,\cV_{\eta_{\ell}(s)})- \bE[\sigma(x,X_{\eta_{\ell}(s)})]}\Big\| \, {\mu}^{Z^{{\ell}}|\F_T^\cV}_{\eta_{\ell}(s)}(dx)\bigg]ds\bigg).
\end{align*}
\end{theorem}
\begin{proof}
To lighten the notation in this proof, we use $t_k$, $\eta(s)$ and $Z$ to denote $t^{\ell}_k$, $\eta_{\ell} (s)$ and $Z^{\ell}$ respectively. First, we observe that 
$$
	|\E[P(y,Z_{s})] - \E[P(y,X_{s})]| \leq \E |\E[P(y,Z_{s})|\F_T^{\cV}] -\E[P(y,X_{s})]|.
$$ 
From definition of $v(\cdot,\cdot)$ in \eqref{eq:defvinteract}, we compute that
\begin{eqnarray}
\E[v_y(0,X_0)]=\int_{\R^d} v_y(0,x) \, \mu_0(dx) &  = & \int_{\R^d}\E \big[ P(y,\cX^{0,x}_t) \big] \, \mu_0(dx) \nonumber \\
& = & \E \big[ \E[P(y,X_t) | X_0  ] \big].  \nonumber 
\end{eqnarray}
The Feynman-Kac theorem, hypothesis \hvb \,  and the fact that $\mu^X_0 = \mu^{Z}_0$ give 
\begin{align*}
	\bE[P(y,Z_{t})|\F_T^{\cV}] - \bE[P(y,X_t)] 
	&= \bE[v_y(t,Z_{t})|\F_T^{\cV}] - \bE[v_y(0, Z_0)] \\
	&= \bE[v_y(t,Z_{t})|\F_T^{\cV}] - \bE[v_y(0, Z_0)|\F_T^{\cV}] \\
	&= \sum_{k=0}^{n-1}\bE\big[v_y(t_{k+1}, Z_{k+1}) -v_y(t_{k}, Z_{k}) \,  \big| \, \F_T^{\cV} \big],
\end{align*} where $n=t/h_{\ell}$ \footnote{For simplicity we assume that $n$ is an integer.}. By It\^o's formula, 
\begin{eqnarray}
	&& \bE[v_y(t,Z_{t})|\F_T^{\cV}] - \bE[v_y(0, Z_0)] \nonumber \\ 
	& = & \sum_{k=0}^{n-1}\bE\bigg[\int_{t_{k}}^{t_{k+1}}\bigg(\partial_t v_y (s, Z_s) +\sum_{j=1}^d \partial_{x_j} v_y  (s,Z_s)\overline{b}_j(Z_{\eta(s)},\cV_{\eta(s)}) \nonumber \\
	&& +\dfrac{1}{2}\smsec\partial^2_{x_i,x_j} v_y  (s, Z_s)\overline{a}_{ij}(Z_{\eta(s)}, \cV_{\eta(s)})\bigg)ds \nonumber \\
	&& + \int_{t_{k}}^{t_{k+1}} \smfir\smfi
	\partial_{x_j} v_y   (s,Z_s)\overline{\sigma}_{ji}(Z_{\eta(s)}, \cV_{\eta(s)})dW_s^{(i)} \, \bigg| \,\F^\cV_T\bigg], \nonumber 
\end{eqnarray}
where $\overline{a}(x,\mu)=\overline{\sigma}(x,\mu)\overline{\sigma}(x,\mu)^T$. Condition (\hvdiffs), as well as hypotheses \hLip, \hLaw \, and  \hvb, along with Lemma \ref{lm:intergrabilitySDErandomcoeffinteract}  and part (a) of Lemma \ref{lem:conmartingaleinteract} (with  the filtration $\{\mathcal{F}_t \}_{t \in [0,T]}$ such that $\mathcal{F}_t = \sigma( \F_T^\cV, \{ W_u \}_{0 \leq u \leq t}, \{ Z_u \}_{0 \leq u \leq t})$) imply that
\begin{equation}\label{eq:proofalphafinalnonint2interact}
	\bE\bigg[\int_{t_{k}}^{t_{k+1}}  \smfir\smfi\partial_{x_j} v_y  (s,Z_s)\overline{\sigma}_{ji}(Z_{\eta(s)}, \cV_{\eta(s)})dW_s^{(i)} \, \bigg| \, \F_T^{\cV}\bigg]=0.
\end{equation}
Subsequently, using the fact that $v(\cdot,\cdot)$ satisfies PDE \eqref{eq:proofalphafinalnonintinteract}, we have
\begin{eqnarray}
	&& \bE[v_y(t,Z_{t})|\F_T^{\cV}] - \bE[v_y(0, Z_0)] \nonumber \\
	& = & \sum_{k=0}^{n-1}\int_{t_{k}}^{t_{k+1}}\bE\bigg[\sum_{j=1}^d\partial_{x_j} v_y  (s,Z_s)(\overline{b}_j(Z_{\eta(s)},\cV_{\eta(s)})-b_j[Z_s, \mu^X_s]) \nonumber \\
	&& +\dfrac{1}{2}\smsec\partial^2_{x_i,x_j} v_y  (s, Z_s)(\overline{a}_{ij}(Z_{\eta(s)}, \cV_{\eta(s)})-a_{ij}[Z_s, \mu^X_s] ) \,  \bigg| \, \F^\cV_T\bigg]ds. \nonumber
\end{eqnarray}
 Hence,
\begin{align*}
	\bE[v_y(t,Z_{t})|\F_T^{\cV}] - \bE[v_y(0, Z_0)]
	&=\sum_{k=0}^{n-1}\int_{t_{k}}^{t_{k+1}}\bE \bigg[ \sum_{i=1}^4 R_i(s) \bigg| \F_T^{\cV} \bigg]ds,
	\end{align*}	
	where 
\begin{align*}
	R_1(s)&\defeq  \smfir\partial_{x_j} v_y  (s,Z_s)(b_j[Z_{\eta(s)}, \mu^X_{\eta(s)}]-b_j[Z_{s}, \mu^X_{s}]) \\
	R_2(s)&\defeq\smfir\partial_{x_j} v_y  (s,Z_s)(\overline{b}_j(Z_{\eta(s)},\cV_{\eta(s)})-b_j[Z_{\eta(s)}, \mu^X_{\eta(s)}])  \\
	R_3(s)&\defeq\frac{1}{2}\smsec\partial^2_{x_i,x_j} v_y (s,Z_s)(a_{ij}[Z_{\eta(s)}, \mu^X_{\eta(s)}]-a_{ij}[Z_{s}, \mu^X_{s}]) \\
	R_4(s)&\defeq \frac{1}{2}\smsec\partial^2_{x_i,x_j} v_y (s,Z_s)(\overline{a}_{ij}(Z_{\eta(s)},\cV_{\eta(s)})-a_{ij}[Z_{\eta(s)}, \mu^X_{\eta(s)}]).
\end{align*}
\paragraph*{Error $R_1$:}
Let $\mathcal{F}^{Z}_T$ be the sigma-algebra generated by $\{Z_t\}_{t \in [0,T]}$.  From part (a) of Lemma \ref{lem:conmartingaleinteract} and the It\^{o}'s formula, we have  \small
{\begin{eqnarray}
&& \bE[R_1(s)|\F^\cV_T] \nonumber \\
&= & \sum_{k=1}^d \bE \bigg[ \partial_{x_k} v_y  (s,Z_{s}) \bE \bigg[ \int^{s}_{\eta(s)} \bigg[ \partial_{u} b_k [Z_{u}, \mu^X_u] + \sum_{i=1}^d \partial_{x_i} b_k [Z_{u}, \mu^X_u] \overline{b}_i (Z_{\eta(u)}, \mathcal{V}_{\eta(u)}) + \nonumber \\
&& + \frac{1}{2} \sum_{i,j=1}^d \partial^2_{x_i,x_j} b_k [Z_{u}, \mu^X_u] \overline{a}_{ij}(Z_{\eta(u)}, \mathcal{V}_{\eta(u)}) \bigg] \,du \,  \bigg| \sigma(\mathcal{F}^{Z}_T,\F^\cV_T) \bigg] \, \bigg| \F^\cV_T \bigg]. \nonumber
\end{eqnarray}}
\normalsize{Condition (\hvdiffs) \, and the conditional Jensen inequality imply that} \small
{\begin{eqnarray}
&&	\bE \big|\bE[R_1(s)|\F^\cV_T] \big| \nonumber \\
& \leq &  c \sum_{k=1}^d \bigg(  \int^{s}_{\eta(s)} \bE\bigg| \partial_{u} b_k [Z_{u}, \mu^X_u]+ \sum_{i=1}^d \partial_{x_i} b_k [Z_{u}, \mu^X_u] \overline{b}_i (Z_{\eta(u)}, \mathcal{V}_{\eta(u)}) + \nonumber \\
    &&    \frac{1}{2} \sum_{i,j=1}^d \partial^2_{x_i,x_j} b_k [Z_{u}, \mu^X_u] \overline{a}_{ij}(Z_{\eta(u)}, \mathcal{V}_{\eta(u)})  \bigg| du\bigg). \label{eq:conditionalcauchy}
\end{eqnarray}}
\normalsize{Using these two bounds along with Lemma \ref{lm:regularityBinteract} and assumption \hvLip, we can see that 
$$
\bE \big|\bE[R_1(s)|\F^\cV_T] \big|  \leq c \bigg( \int_{\eta(s)}^s 1+ \sup_{s' \in [0,t]} \bE |Z_{s'}|^{2} +  \sup_{s' \in [0,t]} \bE \bigg| \int_{\bR^d} |x|^{2} \mathcal{V}_{s'} (dx) \bigg| \, du
\bigg) .
$$}
Assumptions \hLip, \, \hLaw \, and \hvb \,  allow us to conclude that
\[
	\sup_{0\leq s\leq t}\bE|\bE[R_1(s)|\F^\cV_T]|\leq c h_{\ell}.
\]

\paragraph*{Error $R_2$:} Condition (\hvdiffs)\ implies that
$$
\big|\bE[R_2(s)|\F^\cV_T] \big| \leq c\ \bE \big[|b[Z_{\eta(s)}, \mu^X_{\eta(s)}] -\overline{b}(Z_{\eta(s)},\cV_{\eta(s)})| \, \big|\F_T^\cV \big]. \label{R2bound}
$$ 
Using the notation of regular conditional probability measures,
\begin{gather}  \bE |\bE[R_2(s)|\F^\cV_T]|  \leq c \,  \bE \bigg[\int_{\R^d} \big| \bE[b(x,X_{\eta(s)})] - \overline{b}(x,\cV_{\eta(s)}) \big|\, {\mu}^{Z|\F_T^\cV}_{\eta(s)}(dx)\bigg]. \nonumber \end{gather}
Similarly, by the condition on the second-order derivatives from (\hvdiffs), we can establish that \begin{equation} \label{eq:secondderivative1}
\sup_{0\leq s\leq T}\bE|\bE[R_3(s)|\F^\cV_T]|\leq c h_{\ell} \end{equation}
and 
\begin{equation} \label{eq:secondderivative2}  |\bE[R_4(s)|\F^\cV_T]|  \leq  c \,  \bE \big[  \big\| \sigma[Z_{\eta(s)}, \mu^X_{\eta(s)}] - \overline{\sigma}(Z_{\eta(s)},\cV_{\eta(s)})  \big\| \, \big|\F_T^\cV \big] . \end{equation}

\end{proof}
Next, we introduce an artificial process $\bar{Z}^{\ell}$ in order to remove the dependence of $Z^{\ell}$ on $\F^\cV_T$. Note that ${\mu}^{Z^{{\ell}}|\F_T^\cV}_{\eta_{\ell}(s)}$ is a random measure, whereas $\mu^{\bar{Z}^{\ell}}_{\eta_{\ell}(s)}$ is non-random. This is crucial in the iteration that will be discussed in the next section. 

\begin{lemma} \label{lm:weakerror1nonintinteract}
Let $P \in C^{0,2}_{b,b}( \bR^d \times \bR^d, \bR)$ be a Lipschitz continuous function. Assume that \hkreg, \hLaw,  \hvb \, and \hvLip~hold. 
Then there exists a constant $c$ $($independent of the choices of $L$ and $N_1, \ldots, N_L)$ such that for each $t\in[0,T]$, $\ell \in \{0, \ldots, L \}$ and $x\in\R^d$,
\begin{eqnarray}
&& \sup_{0\leq s\leq t} \bE \Big[ |\bE[P(x,Z_{s}^{\ell})  |\F_T^\cV ]-\bE[P(x,X_{s})]|^2 \Big] \nonumber \\
&\leq & c\bigg(h_{\ell}^2+\int_{0}^{t}\bigg[ \int_{\R^d} \bE|\overline{b}(x,\cV_{\eta_{\ell}(s)})- \bE[b(x,X_{\eta_{\ell}(s)})]|^2  \mu^{\bar{Z}^{\ell}}_{\eta_{\ell}(s)} (dx)\bigg]ds \nonumber \\
&& + \int_{0}^{t}\bigg[\int_{\R^d} \bE \Big\|\overline{\sigma}(x,\cV_{\eta_{\ell}(s)})- \bE[ \sigma (x,X_{\eta_{\ell}(s)})] \Big\|^2 \mu^{\bar{Z}^{\ell}}_{\eta_{\ell}(s)} (dx)\bigg]ds \bigg), \nonumber 
\end{eqnarray}
where $\bar{Z}^{{\ell}}$ is a process defined by
$$ d \bar{Z}^{{\ell}}_t = \int_{\bR^d} b (\bar{Z}^{{\ell}}_{\eta_{\ell}(t)}, y)  \, \mu^{X}_{\eta_{\ell}(t)} (dy)  \, dt + \int_{\bR^d} \sigma (\bar{Z}^{{\ell}}_{\eta_{\ell}(t)} ,y) \,  \mu^{X}_{\eta_{\ell}(t)} (dy) \, dW_t.$$
\end{lemma}
\begin{proof}
\normalsize{As in the proof of Theorem \ref{lm:conditionalwknonintmaininteract}, we use $\eta(s)$, $Z$ and $\bar{Z}$  to denote $\eta_{\ell}(s)$, $Z^{\ell}$ and $\bar{Z}^{\ell}$ respectively. By \hLip \, and \hvLip, 
\begin{align} \label{R2decomposition}
& \bE \Big[ \Big| \big( b[Z_{\eta(s)}, \mu^X_{\eta(s)}]-\overline{b}(Z_{\eta(s)},\cV_{\eta(s)}) \big) - \big( b[\bar{Z}_{\eta(s)}, \mu^X_{\eta(s)}]-\overline{b}({\bar{Z}}_{\eta(s)},\cV_{\eta(s)}) \big) \Big|^2 \, \Big|\F_T^\cV \Big]  \nonumber\\
&  \leq  c \, \bE \big[ \big| Z_{\eta(s)} - {\bar{Z}}_{\eta(s)}\big|^2 \, \big| \F_T^\cV \big]. 
\end{align}
We further decompose the error as follows.}
\footnotesize {
\begin{eqnarray}
\bE \big[ \big| Z_{\eta(s)} - {\bar{Z}}_{\eta(s)}\big|^2 \, \big| \F_T^\cV \big] & \leq &  2 \Bigg( \bE \bigg[ \bigg| \int_0^s  \bigg( b[\bar{Z}_{\eta(u)}, \mu^X_{\eta(u)}] - \overline{b} \big( {{Z}}_{\eta(u)},  \cV_{\eta(u)} \big) \bigg) \,du \bigg|^2  \, \bigg| \F_T^\cV  \bigg]  \nonumber \\
&& + \bE \bigg[ \bigg| \int_0^s  \bigg(  \sigma[\bar{Z}_{\eta(u)}, \mu^X_{\eta(u)}]  - \overline{\sigma}  \big( {{Z}}_{\eta(u)},  \cV_{\eta(u)} \big) \bigg) \,dW_u \bigg|^2  \, \bigg| \F_T^\cV  \bigg] \Bigg)  \nonumber \\
& =:& 2(R_{21} (s) + R_{22} (s)). \nonumber \end{eqnarray}}
\normalsize{By the conditional Fubini's theorem and the Cauchy-Schwarz inequality, there exists a constant $K>0$ such that} \footnotesize{
\begin{eqnarray}
&& R_{21} (s) \nonumber \\
& \leq & c \,  \bigg( \int_0^s \bE \bigg[ \bigg| b[\bar{Z}_{\eta(u)}, \mu^X_{\eta(u)}] - \overline{b} \big( {\bar{Z}}_{\eta(u)},  \cV_{\eta(u)} \big) \bigg|^2 \bigg| \F_T^\cV  \bigg] \nonumber \\
&& + \bE \bigg[ \bigg| \overline{b} \big( {\bar{Z}}_{\eta(u)},  \cV_{\eta(u)} \big)  - \overline{b} \big( {{Z}}_{\eta(u)},  \cV_{\eta(u)} \big) \bigg|^2 \bigg| \F_T^\cV  \bigg]  \,du \bigg) \nonumber \\
&  \leq & c \,  \bigg( \int_0^s \bE \bigg[ \bigg| b[\bar{Z}_{\eta(u)}, \mu^X_{\eta(u)}] - \overline{b} \big( {\bar{Z}}_{\eta(u)},  \cV_{\eta(u)} \big) \bigg|^2 \bigg| \F_T^\cV  \bigg] + \bE \big[ \big| Z_{\eta(u)} - {\bar{Z}}_{\eta(u)}\big|^2 \big| \F_T^\cV \big] \,du \bigg), \nonumber
\end{eqnarray}}
\normalsize{where assumption \hvLip~ is used in the final inequality.
Since $\bar{Z}$ is independent of $ \F_T^\cV $ and that $\mu^{X}_{\eta(u)}$ is a non-random measure, we use the properties of regular conditional distributions as outlined in Theorem 7.1 of \cite{parthasarathy1967probability}  to prove that for each $\omega \in \Omega$,} \footnotesize{
\begin{eqnarray}
&&\Bigg( \bE \bigg[ \bigg| b[\bar{Z}_{\eta(u)}, \mu^X_{\eta(u)}] - \overline{b} \big( {\bar{Z}}_{\eta(u)}, \cV_{\eta(u)} \big) \bigg|^2 \bigg| \F_T^\cV  \bigg] \Bigg)(\omega) \nonumber \\
& = &    \int_{\bR^d} \bigg| b[x, \mu^X_{\eta(u)}] - \overline{b} ( x,  \cV_{\eta(u)} (\omega)) \bigg|^2 \, \mu^{\bar{Z}}_{\eta(u)}  (dx).  \nonumber 
\end{eqnarray}}
\normalsize{Therefore,} \footnotesize{
$$ R_{21} (s) \leq c \bigg( \int_0^s \bigg[ \bE \big[ \big| Z_{\eta(u)} - {\bar{Z}}_{\eta(u)}\big|^2 \big| \F_T^\cV \big] + \int_{\bR^d} \bigg| b[x, \mu^X_{\eta(u)}] - \overline{b} ( x,  \cV_{\eta(u)} ) \bigg|^2 \, \mu^{\bar{Z}}_{\eta(u)}  (dx) \bigg] \, du \bigg). $$}
\normalsize{We proceed similarly as $R_{22} (s)$ and apply part (b) of Lemma \ref{lem:conmartingaleinteract} (with  the filtration $\{\mathcal{F}_t \}_{t \in [0,T]}$ such that $\mathcal{F}_t = \sigma( \F_T^\cV, \{ W_u \}_{0 \leq u \leq t}, Z_0)$) to get} \footnotesize{
$$ R_{22} (s) \leq c \bigg( \int_0^s \bigg[ \bE \big[ \big| Z_{\eta(u)} - {\bar{Z}}_{\eta(u)}\big|^2 \big| \F_T^\cV \big] + \int_{\bR^d} \bigg\| \sigma[x, \mu^X_{\eta(u)}] - \overline{\sigma} ( x,  \cV_{\eta(u)} ) \bigg\|^2 \, \mu^{\bar{Z}}_{\eta(u)}  (dx) \bigg] \, du \bigg). $$}
\normalsize{Combining both bounds gives} \footnotesize{
\begin{eqnarray}
 \bE \big[ \big| Z_{\eta(s)} - {\bar{Z}}_{\eta(s)}\big|^2 \big| \F_T^\cV \big] & \leq &  c \bigg( \int_0^s \bigg[ \bE \big[ \big| Z_{\eta(u)} - {\bar{Z}}_{\eta(u)}\big|^2 \big| \F_T^\cV \big] \nonumber \\
 && + \int_{\bR^d} \bigg| b[x, \mu^X_{\eta(u)}] - \overline{b} ( x,  \cV_{\eta(u)} ) \bigg|^2 \, \mu^{\bar{Z}}_{\eta(u)}  (dx)   \nonumber \\
&& + \int_{\bR^d} \bigg\| \sigma[x, \mu^X_{\eta(u)}] - \overline{\sigma}  ( x,  \cV_{\eta(u)} ) \bigg\|^2 \, \mu^{\bar{Z}}_{\eta(u)}  (dx) \bigg] \, du \bigg), 
 \nonumber
\end{eqnarray}}
\normalsize{for any $s \in [0,t]$. 
By Gronwall's lemma and integration  from $0$ to $t$ in time, we obtain that} \footnotesize{ 
\begin{eqnarray}
\int_0^t \bE \big[ \big| Z_{\eta(s)} - {\bar{Z}}_{\eta(s)}\big|^2 \big| \F_T^\cV \big] \, ds & \leq & c \bigg( \int_0^t \bigg[ \int_{\bR^d} \bigg| b[x, \mu^X_{\eta(s)}] - \overline{b} ( x,  \cV_{\eta(s)} ) \bigg|^2 \, \mu^{\bar{Z}}_{\eta(s)}  (dx) 
 \nonumber\\
 && + \int_{\bR^d} \bigg\| \sigma[x, \mu^X_{\eta(s)}]  - \overline{\sigma}  ( x,  \cV_{\eta(s)} ) \bigg\|^2 \, \mu^{\bar{Z}}_{\eta(s)} (dx) \bigg] \, ds \bigg). \nonumber 
\end{eqnarray}}
\normalsize{By (\ref{R2bound}) and (\ref{R2decomposition}), it is clear that} \footnotesize{
\begin{eqnarray}
\int_0^t |\bE[R_2(s)|\F_T^V]|^2 \,ds & \leq & c \bigg( \int_0^t \bE \big[ \big| Z_{\eta(s)} - {\bar{Z}}_{\eta(s)}\big|^2 \big| \F_T^\cV \big] \nonumber \\
&& + \bE \big[|b[\bar{Z}_{\eta(s)}, \mu^X_{\eta(s)}] -\overline{b}({\bar{Z}}_{\eta(s)},\cV_{\eta(s)}) |^2 \, \big| \F_T^\cV \big] \, ds \bigg). \nonumber 
\end{eqnarray} }
\normalsize{This shows that} \footnotesize{
\begin{eqnarray}
\int_0^t |\bE[R_2(s)|\F_T^{\cV}]|^2 \,ds & \leq & c \bigg( \int_0^t \bigg[ \int_{\bR^d} \bigg| b[x, \mu^X_{\eta(s)}]  - \overline{b} ( x,  \cV_{\eta(s)} ) \bigg|^2 \, \mu^{\bar{Z}}_{\eta(s)} (dx) 
 \nonumber\\
 && + \int_{\bR^d} \bigg\| \sigma[x, \mu^X_{\eta(s)}]  - \overline{\sigma}  ( x,  \cV_{\eta(s)} ) \bigg\|^2 \, \mu^{\bar{Z}}_{\eta(s)} (dx) \bigg] \, ds \bigg). \nonumber 
\end{eqnarray} }
\normalsize{We repeat the same argument for $R_4(s)$ and conclude that} \footnotesize{
\begin{eqnarray}
\int_0^t |\bE[R_4(s)|\F_T^{\cV}]|^2 \,ds & \leq & c \bigg( \int_0^t \bigg[ \int_{\bR^d} \bigg| b[x, \mu^X_{\eta(s)}]  - \overline{b} ( x,  \cV_{\eta(s)} ) \bigg|^2 \, \mu^{\bar{Z}}_{\eta(s)} (dx) 
 \nonumber\\
 && + \int_{\bR^d} \bigg\| \sigma[x, \mu^X_{\eta(s)}]  - \overline{\sigma}  ( x,  \cV_{\eta(s)} ) \bigg\|^2 \, \mu^{\bar{Z}}_{\eta(s)}(dx) \bigg] \, ds \bigg). \nonumber 
\end{eqnarray}}
\end{proof}
\section{Iteration of the MLMC algorithm}

\subsection{Interacting kernels}
Fix $m \geq 1$ and correspond each particle $Z^{i, \ell}$ in the abstract framework with $Y^{i,m, \ell}$ defined in \eqref{eq:contiousYinteract} and   $\F_T^\cV$  with the sigma-algebra $\mathcal{F}^{m-1}$ generated by all the particles $Y^{i,m-1,\ell}$ in the $(m-1)$th Picard step, $0 \leq \ell \leq L,1 \leq i \leq N_{m-1,\ell}$.  We set $\cV_t :=  \cM^{(m-1)}_t$ (defined in \eqref{eq:defMLMCop}),  $\overline{b}(x,\mu):= b[x,\mu]$ and $\overline{\sigma}(x,\mu):=\sigma[x,\mu]$, so that
\[
 \overline{b}(x, \mathcal{M}^{(m-1)}_t ) = \langle \mathcal{M}^{(m-1)}_t,   b(x, \cdot) \rangle  \quad \text{and} \quad \overline{\sigma}(x, \mathcal{M}^{(m-1)}_t ) = \langle \mathcal{M}^{(m-1)}_t,   \sigma(x, \cdot) \rangle, 
\]
for each $x \in \bR^d$. The measure $\mathcal{M}^{(m-1)}$ satisfies the independence criterion in \hvb, since $\{Y^{m-1}\} \perp (W^{m},Z_0^m)$.  The criteria \hvb, \hvreg \,  and \hvLip~ are verified below. 

In the results of this section, $c$ denotes a generic constant that depends on $T$, but not on $m$,$\ell$ or $N_{m,\ell}$.

\begin{lemma}[Verification of \hvLip] \label{lm:Regularitycoeffiinteract}
Assume \hLip \, and \hLaw. Then, for each $t\in[0,T]$, there exists a constant $c$ such that for all $x_1,x_2\in\R^d$
\[
	|   \langle \mathcal{M}^{(m-1)}_{t} ,  b( x_1, \cdot)  - b(x_2, \cdot) \rangle| +\norm{ \langle \mathcal{M}^{(m-1)}_{t} ,  \sigma( x_1, \cdot)  - \sigma(x_2, \cdot) \rangle }\leq c|x_1-x_2|, 
\]
\[
	| \langle \mathcal{M}^{(m-1)}_{t}b(x_1, \cdot)  \rangle | + \norm{\langle \mathcal{M}^{(m-1)}_{t}\sigma(x_1, \cdot)  \rangle} \leq c \bigg( 1+ |x|+  \bigg| \int_\Rd |y|\mathcal{M}^{(m-1)}_{t}(dy)  \bigg|\bigg)\,.
	\]
\end{lemma}
\begin{proof}
For any $t\in[0,T]$ and $x_1,x_2\in \R^d$, by the definition of $\cM^{(m-1)}_t$, 
\begin{eqnarray}
	&& \Big|\langle \mathcal{M}^{(m-1)}_t,   b(x_1, \cdot) \rangle   - \langle \mathcal{M}^{(m-1)}_t,   b(x_2, \cdot) \rangle \Big| \nonumber \\
	& = &  \Bigg| \sum_{\ell=1}^{L} \frac{1}{N_{m-1,\ell}} \sum_{i=1}^{N_{m-1,\ell} }
 \bigg[\bigg(\frac{ t- \eta_{\ell} (t)}{h_\ell}\bigg)\cdot \left( b(x_1,Y^{i,m-1,\ell}_{\eta_{\ell}{(t)}+h_\ell})- b(x_2,Y^{i,m-1,\ell}_{\eta_{\ell}{(t)}+h_\ell})\right)\nonumber \\
 && +\bigg(1-\frac{ t- \eta_{\ell} (t)}{h_\ell}\bigg)\cdot \left( b(x_1,Y^{i,m-1,\ell}_{\eta_{\ell}{(t)}})- b(x_2,Y^{i,m-1,\ell}_{\eta_{\ell}{(t)}})\right)  \nonumber \\
 && -\bigg(\frac{ t- \eta_{\ell-1} (t)}{h_{\ell-1}}\bigg)\cdot \left( b(x_1,Y^{i,m-1,\ell-1}_{\eta_{\ell-1} {(t)}+h_{\ell-1}})- b(x_2,Y^{i,m-1,\ell-1}_{\eta_{\ell-1}{(t)}+h_{\ell-1}})\right) \nonumber \\
 &&  -\bigg(1-\frac{ t- \eta_{\ell-1} (t)}{h_{\ell-1}}\bigg)\cdot \left( b(x_1,Y^{i,m-1,\ell-1}_{\eta_{\ell-1}{(t)}})- b(x_2,Y^{i,m-1,\ell-1}_{\eta_{\ell-1}{(t)}})\right) \bigg]  \nonumber \\
 &&  + \frac{1}{N_{m-1,0}} \sum_{i=1}^{N_{m-1,0} }
 \bigg[\bigg(\frac{ t- \eta_{0} (t)}{h_0}\bigg)\cdot \left( b(x_1,Y^{i,m-1,0}_{\eta_{0}{(t)}+h_0})- b(x_2,Y^{i,m-1,0}_{\eta_{0}{(t)}+h_0})\right)\nonumber \\
 && +\bigg(1-\frac{ t- \eta_{0} (t)}{h_0}\bigg)\cdot \left( b(x_1,Y^{i,m-1,0}_{\eta_{0}{(t)}})- b(x_2,Y^{i,m-1,0}_{\eta_{0}(t)})\right) \bigg] \Bigg|.  \nonumber
\end{eqnarray}
The required bounds follow from  \hLip \,.
The corresponding estimates for $\norm{\overline{\sigma}(x_1, \cV_{\eta(t)} ) - \overline{\sigma}(x_2, \cV_{\eta(t)})}$ and $\norm{\overline{\sigma}(x_1, \cV_{\eta(t)})}$ can be obtained in a similar way and are hence omitted.  
\end{proof}

\begin{lemma}[Verification of \hvb] \label{lm:iteratedmlmcintegrabilityinteract}
Assume \hLip \, and \hLaw.  Then for any $p\geq 2$, there exists a constant $c$ such that  \[
        \sup_{n \in \mathbb{N} \cup \{0 \}}\sup_{t\in[0,T]}\bE \bigg| \int_{\bR^d } |x|^p {\cM}^{(n)}_t (dx) \bigg| \leq c.
\]
\end{lemma}
\begin{proof}
 For simplicity of notation, we rewrite 
\begin{align*}
     \int_{\bR^d } |x|^p {\cM}^{(n)}_t (dx)   \defeq \frac{1}{ N_{0} } \sum_{i=1}^{N_{0}} P_{t}^{i,0}
+ \sum_{\ell=1}^{L} \frac{1}{N_{\ell}} \sum_{i=1}^{N_{\ell} }
 \left( P_{t}^{i,\ell} - P_{t}^{i,\ell-1}\right),
\end{align*}
where 
$$ P^{i,\ell}_t = \bigg( \frac{t- \eta_{\ell} (t)}{h_{\ell}} \bigg)  \big| Y^{i,n, \ell}_{\eta_{\ell} (t) + h_{\ell}} \big|^p + \bigg( 1- \frac{t- \eta_{\ell} (t)}{h_{\ell}} \bigg)  \big| Y^{i,n, \ell}_{\eta_{\ell} (t) } \big|^p.$$
We first fix $\ell \in \{1, \ldots , L \}$ and define
\[
        \Delta_t^{i,\ell} := \bE|P_{t}^{i,\ell} - P_{t}^{i,\ell-1} |, \quad i \in \{ 1, \ldots, N_{\ell} \}.
\]
By exchangeability, there exists a constant $c$ (independent of the Picard step $n$) such that
\[
        \bE[|\Delta_{t}^{i,\ell}|]\leq c\sum_{\ell'=\ell-1}^\ell(\bE |Y_{\eta_{\ell'}(t)}^{1, n, \ell'}|^p+\bE |Y_{\eta_{\ell'}(t)+h_{\ell'}}^{1, n, \ell'}|^p).
\]
By the triangle inequality,
\begin{align*}
\bE \bigg|\frac{1}{N_{\ell}}\sum^{N_{\ell}}_{i=1}\Delta_{t}^{i,\ell} \bigg|\leq N_{\ell}^{-1}\sum^{N_{\ell}}_{i=1}\bE|\Delta_{t}^{i,\ell}| \leq  c\sum_{\ell'=\ell-1}^\ell\bigg( \bE |Y_{\eta_{\ell'}(t)}^{1,n,\ell'}|^p+\bE |Y_{\eta_{\ell'}(t)+h_{\ell'}}^{1,n,\ell'}|^p\bigg).
\end{align*}
Similarly, we can show that
\[
        \bE \bigg|\frac{1}{N_0} \sum_{i=1}^{N_0} P^{i,0}_t \bigg| \leq c \bigg( \bE |Y_{\eta_{0}(t)}^{1,n,0}|^p+\bE |Y_{\eta_{0}(t)+h_{0}}^{1,n,0}|^p\bigg).
\]
Note that
\begin{eqnarray}
\bE \bigg| \int_{\bR^d } |x|^p {\cM}^{(n)}_t (dx)  \bigg|&  \leq & \bE \bigg| \frac{1}{N_0} \sum_{i=1}^{N_0} P^{i,0}_t + \sum_{\ell=1}^{L}\frac{1}{N_{\ell}}\sum^{N_{\ell}}_{i=1}\Delta_{t}^{i,\ell} \bigg| \nonumber \\
& \leq & c \sum_{\ell=0}^L \bigg( \bE |Y_{\eta_{\ell}(t)}^{1,n,\ell}|^p+\bE |Y_{\eta_{\ell}(t)+h_{\ell}}^{1,n,\ell}|^p\bigg). \nonumber 
\end{eqnarray}
We can see from the proof of Lemma \ref{lm:Regularitycoeffiinteract} that the constant $c$ in Lemma \ref{lm:intergrabilitySDErandomcoeffinteract} does not depend on the particular Picard step. Therefore, by Lemma \ref{lm:intergrabilitySDErandomcoeffinteract},  $$ \sup_{0\leq t\leq T} \bE \bigg| \int_{\bR^d } |x|^p {\cM}^{(n)}_t (dx) \bigg| \leq c\bigg( 1+ \int_0^T \sup_{0\leq u\leq s} \bE \bigg| \int_{\bR^d } |x|^p {\cM}^{(n-1)}_u (dx) \bigg| \, ds \bigg).
$$ 
By iteration, we conclude that
\begin{eqnarray}
         \sup_{0\leq t\leq T} \bE \bigg| \int_{\bR^d } |x|^p {\cM}^{(n)}_t (dx) \bigg| & \leq & \sum_{r=0}^{n-1} \frac{(cT)^r}{r!} +  \sup_{0\leq t\leq T} \bE \bigg| \int_{\bR^d } |x|^p {\cM}^{(0)}_t (dx) \bigg| \frac{(cT)^n}{n!} \nonumber \\ & \leq & e^{cT} \bigg( 1+ \sup_{0\leq t\leq T} \bE \bigg| \int_{\bR^d } |x|^p {\cM}^{(0)}_t (dx) \bigg| \bigg) < +\infty. \nonumber
\end{eqnarray}
\end{proof}

\begin{lemma}[Verification of \hvreg] \label{prop:nonintcomplexitywemfZtinteract}
Assume \hLip \, and \hLaw. Given any Lipschitz continuous function $C^{0,2}_{b,b} \ni P:\R^d\times\R^d\rightarrow\R$ and $n \in \mathbb{N} \cup \{0 \}$, there exists a constant $c$ such that 
\begin{equation}\label{eq:nonintsterrormeanfZt}
	\bE\bigg| \langle \mathcal{M}^{(n)}_t,   P(x, \cdot) \rangle -\langle \mathcal{M}^{(n)}_s,   P(x, \cdot) \rangle \bigg|^2 \leq c(t-s),
\end{equation}
for any $x \in \bR^d$ and $0 \leq s \leq t \leq T$.
\end{lemma}

\begin{proof}\label{pf:nonintcomplexitywemfZt}
When analysing the regularity of MLMC measure \eqref{eq:nonintsterrormeanfZt} one needs to pay attention to the interpolation in time that we used. Pick any $\ell^*\in\{0,1,2,\ldots L\}$.  For simplicity of notation, we rewrite $ \langle \mathcal{M}^{(n)}_t,   P(x, \cdot) \rangle $  as
\begin{align}\label{eq:mtg}
       \langle \mathcal{M}^{(n)}_t,   P(x, \cdot) \rangle \defeq \frac{1}{ N_{n,0} } \sum_{i=1}^{N_{n,0}} P_{t}^{i,0}
+ \sum_{\ell=1}^{L} \frac{1}{N_{n,\ell}} \sum_{i=1}^{N_{n,\ell} }
 \left( P_{t}^{i,\ell} - P_{t}^{i,\ell-1}\right),
\end{align}
where 
$$ P^{i,\ell}_t = \bigg( \frac{t- \eta_{\ell} (t)}{h_{\ell}} \bigg)  P\big( x, Y^{i,n, \ell}_{\eta_{\ell} (t) + h_{\ell}} \big) + \bigg( 1- \frac{t- \eta_{\ell} (t)}{h_{\ell}} \bigg)  P \big( x,Y^{i,n, \ell}_{\eta_{\ell} (t) } \big).$$
Given any $k\in\{0,1,\ldots,2^{L}-1\}$, we compute \small {
\begin{eqnarray*}
&& \langle \mathcal{M}^{(n)}_{t_{k+1}^\ells},   P(x, \cdot) \rangle  -  \langle \mathcal{M}^{(n)}_{t_{k}^\ells},   P(x, \cdot) \rangle \nonumber \\
&=& \frac{1}{ N_{n,0} } \sum_{i=1}^{N_{n,0}} (P_{t_{k+1}^\ells}^{i,0} -P_{t_k^\ells}^{i,0}) + \sum_{\ell=1}^{L} \frac{1}{N_{n,\ell}} \sum_{i=1}^{N_{n,\ell} }
 \bigg( (P_{t_{k+1}^\ells}^{i,\ell} -P_{t_k^\ells}^{i,\ell})- (P_{t_{k+1}^\ells}^{i,\ell-1}-P_{t_k^\ells}^{i,\ell-1})\bigg) .
\end{eqnarray*}}
\normalsize{Thus, we only need to consider} $P_{t_{k+1}^\ells}^{i,\ell} - P_{t_k^\ells}^{i,\ell}$, for each $\ell\in\{0,1,\ldots,L\}.$ There are two cases depending on the value of $\ell$: 
$\ell<\ells$ and $\ell\geq\ells$.

\textit{For levels $\ell<\ells$, at least one of $P_{t_{k+1}^\ells}^{i,\ell}$ and $P_{t_{k}^\ells}^{i,\ell}$ is an interpolated value}. Then there exist a unique $s\in\{0,1,\ldots,2^\ell-1\}$ (chosen such that $\eta_{\ell}(t^{\ell^{*}}_k) = t^{\ell}_s$) and constants $\lambda \in (0,1-\frac{h_\ells}{h_{\ell}}]$ and $\tilde{\lambda}$, given by
\[
        \lambda = \dfrac{t_k^\ells-t_s^\ell}{h_{\ell}}~\text{and}~ \tilde{\lambda} =  \dfrac{t_{k+1}^\ells-t_s^\ell}{h_{\ell}},
\]
such that
\begin{gather*}
        P_{t_{k}^\ells}^{i,\ell}=(1-\lambda)P(x, Y_{t_{s}^\ell}^{i,n,\ell})+  \lambda P(x,Y_{t_{s+1}^\ell}^{i,n,\ell}) ~\text{and}~ P_{t_{k+1}^\ells}^{i,\ell} = (1-\tilde\lambda)P(x, Y_{t_{s}^\ell}^{i,n,\ell} )+  \tilde\lambda P(x,Y_{t_{s+1}^\ell}^{i,n,\ell}).
\end{gather*}
Note that $\tilde{\lambda}-\lambda=\dfrac{h_{\ells}}{h_{\ell}}$. By taking the difference between $P_{t_{k+1}^\ells}^{i,\ell}$ and $P_{t_{k}^\ells}^{i,\ell}$, we compute that
\begin{align}\label{eq:caseintnon1Ztinteract}
P_{t_{k+1}^\ells}^{i,\ell} -P_{t_{k}^\ells}^{i,\ell} =\dfrac{h_{\ells}}{h_{\ell}} (P(x,Y_{t_{s+1}^{\ell}}^{i,n,\ell})-P(x,Y_{t_{s}^{\ell}}^{i,n,\ell})).
\end{align}

 \textit{For levels $\ell\geq\ells$, both of them are not interpolated}. This gives
\begin{align}\label{eq:caseintnon3Ztinteract}
        P_{t_{k+1}^\ells}^{i,\ell}-P_{t_{k}^\ells}^{i,\ell} = P(x,Y_{t_{k+1}^{\ell^*}}^{i,n,\ell})-P(x,Y_{t_{k}^{\ells}}^{i,n,\ell}).
\end{align}
By Lemmas \ref{lm:iteratedmlmcintegrabilityinteract} and \ref{lm:Regularitycoeffiinteract}, the hypotheses of Lemma \ref{lm:regularityYinteract} are satisfied. By applying Lemma \ref{lm:regularityYinteract} to \eqref{eq:caseintnon1Ztinteract} and \eqref{eq:caseintnon3Ztinteract} along with the global Lipschitz property of $P$, we have
\begin{align*}
\bE|P_{t_{k+1}^\ells}^{i,\ell} -P_{t_k^\ells}^{i,\ell}|^2\leq c h_{\ells}\quad\forall\ell\in\{0,1,\ldots,L\}.
\end{align*}
This shows that \footnotesize{
\begin{eqnarray}
&&\bE\bigg|\langle \mathcal{M}^{(n)}_{t_{k+1}^\ells},   P(x, \cdot) \rangle  -  \langle \mathcal{M}^{(n)}_{t_{k}^\ells},   P(x, \cdot) \rangle \bigg|^2 \nonumber \\
& \leq & \frac{1}{ N_{n,0} } \sum_{i=1}^{N_{n,0}} \bE|P_{t_{k+1}^\ells}^{i,0}-P_{t_{k}^\ells}^{i,0}|^2 + \sum_{\ell=1}^{L} \frac{2}{N_{n,\ell}} \sum_{i=1}^{N_{n,\ell} }
 \bigg( \bE|P_{t_{k+1}^\ells}^{i,\ell} -P_{t_{k}^\ells}^{i,\ell}|^2+ \bE|P_{t_{k+1}^\ells}^{i,\ell-1}-P_{t_{k}^\ells}^{i,\ell-1}|^2\bigg) \nonumber \\
 &\leq &  c h_\ells. \nonumber
\end{eqnarray}}
\normalsize{The proof is complete by replacing $s$ and $t$ by $\eta_L(s)$ and $\eta_L(t)$ respectively if any of them (or both) does not belong to $\Pi^L$. }
\end{proof}

Lemma \ref{lm:c1nonintinteract} below gives a decomposition of MSE (mean-square-error) for MLMC along one iteration of the particle system \eqref{eq:contiousYinteract}. 

\begin{lemma}\label{lm:c1nonintinteract}
Assume  \hkreg \, and \hLaw . Let $P \in C^{0,2}_{b,b}( \bR^d \times \bR^d, \bR)$ be a Lipschitz continuous function. 
Let
\[
	MSE^{(m)}_{t} \big( P (x, \cdot) \big) \defeq \bE\bigg[ \Big( \bE[P(x,X_{t})]- \langle \mathcal{M}^{(m)}_t,  P(x,\cdot) \rangle  \Big)^2\bigg] ,\quad t\in[0,T].
\]
Then, there exists a constant $c>0$ $($independent of the choices of $m$, $L$ and $(N_{m,\ell})_{ 0 \leq \ell \leq L})$ such that for every $t \in [0,T]$, \small{
\begin{eqnarray}
&&	 \int_{\bR^d} MSE^{(m)}_{\eta_L(t)} \big( P (x, \cdot) \big) \, \, \mu^{\bar{Z}^L}_{\eta_L(t)} (dx) \nonumber\\
&\leq & c  \bigg(h_L^2+ \int_{0}^{t}  \bigg[ \int_{ \bR^d}  \bE \Big|\langle \mathcal{M}^{(m-1)}_{\eta_L(s)},  b(x,\cdot) \rangle  - \bE[b(x,X_{\eta_L(s)})] \Big|^2 \mu^{\bar{Z}^L}_{\eta_L(s)} (dx) \bigg] \,ds \nonumber \\
	&& +  \int_0^t \bigg[ \int_{\bR^d}  \bE \Big\|\langle \mathcal{M}^{(m-1)}_{\eta_L(s)},  \sigma(x,\cdot) \rangle  - \bE[\sigma(x,X_{\eta_L(s)})] \Big\|^2 \mu^{\bar{Z}^L}_{\eta_L(s)}(dx) \bigg] ds +  \sum_{\ell=0}^L \frac{h_{\ell}}{N_{m,\ell}}\bigg). \nonumber
\end{eqnarray}}
\normalsize{Furthermore, if we assume that the functions $b$ and $\sigma$ are both bounded, then there exists a constant $c>0$ $($independent of the choices of $m$, $L$ and $(N_{m,\ell})_{0 \leq \ell \leq L})$ such that for every $t \in [0,T]$,}
\begin{eqnarray}
&&	 \sup_{x \in \bR^d} MSE^{(m)}_{\eta_L(t)} \big( P (x, \cdot) \big) \nonumber\\
&\leq & c  \bigg(h_L^2+ \int_{0}^{t}  \bigg[ \sup_{ x \in \bR^d}  \bE \Big|\langle \mathcal{M}^{(m-1)}_{\eta_L(s)},  b(x,\cdot) \rangle  - \bE[b(x,X_{\eta_L(s)})] \Big|^2  \bigg] \,ds \nonumber \\
	&& +  \int_0^t \bigg[ \sup_{x \in \bR^d}  \bE \Big\|\langle \mathcal{M}^{(m-1)}_{\eta_L(s)},  \sigma(x,\cdot) \rangle  - \bE[\sigma(x,X_{\eta_L(s)})] \Big\|^2  \bigg] ds +  \sum_{\ell=0}^L \frac{h_{\ell}}{N_{m,\ell}}\bigg). \nonumber
\end{eqnarray}
\end{lemma}
\begin{proof}
For $x\in\R^d$ and $t\in[0,T]$, we consider
\begin{eqnarray}
     &&  \bE\bigg[\Big( \bE[P(x,X_{\eta_L(t)})]- \langle \mathcal{M}^{(m)}_{\eta_L(t)},  P(x,\cdot) \rangle  \Big)^2\bigg] \nonumber \\
& = & \bE\bigg[\bigg( \bE[P(x,X_{\eta_L(t)})]-\bE \bigg[\langle \mathcal{M}^{(m)}_{\eta_L(t)},  P(x,\cdot) \rangle  \bigg|\F^{m-1} \bigg] \nonumber \\
&& +\bE\bigg[\langle \mathcal{M}^{(m)}_{\eta_L(t)},  P(x,\cdot) \rangle  \bigg|\F^{m-1} \bigg] -\langle \mathcal{M}^{(m)}_{\eta_L(t)},  P(x,\cdot) \rangle  \bigg)^2\bigg]. \nonumber
\end{eqnarray}
Observe that
\begin{eqnarray} 
&& MSE^{(m)}_{\eta_L(t)} \big( P (x, \cdot) \big) \nonumber \\
&= & \bE \bigg[ \bigg(\bE[P(x,X_{\eta_L(t)})]-\bE[P(x,Y_{\eta_L(t)}^{1,m,L})|\F^{m-1}]\bigg)^2 \bigg] \nonumber \\
&& +\bE \bigg[ \bigg( \bE\bigg[\langle \mathcal{M}^{(m)}_{\eta_L(t)},  P(x,\cdot) \rangle  \bigg|\F^{m-1} \bigg] -\langle \mathcal{M}^{(m)}_{\eta_L(t)},  P(x,\cdot) \rangle \bigg)^2\bigg] , \label{eq:c1ineq1interact}
\end{eqnarray}
as $ \bE\bigg[\langle \mathcal{M}^{(m)}_{\eta_L(t)},  P(x,\cdot) \rangle  \bigg|\F^{m-1} \bigg] = \bE[P(x,Y_{\eta_L(t)}^{1,m,L})|\F^{m-1}]$ by exchangeability.
Next, from Lemma \ref{lm:weakerror1nonintinteract}, there exists a constant $c$ such that

\begin{eqnarray}
&&	\bE \bigg[ \bigg(\bE[P(x,X_{\eta_L(t)})]-\bE[P(x,Y_{\eta_L(t)}^{1,m,L})|\F^{m-1}]\bigg)^2 \bigg] \nonumber \\
&\leq & c\ \bigg(h_L^2+\int_{0}^{t}  \bigg[ \int_{ \bR^d}  \bE \Big|\langle \mathcal{M}^{(m-1)}_{\eta_L(s)},  b(x,\cdot) \rangle - \bE[b(x,X_{\eta_L(s)})] \Big|^2 \mu^{\bar{Z}^L}_{\eta_L(s)} (dx) \bigg] \,ds \nonumber \\
&& + \int_0^t \bigg[ \int_{\bR^d}  \bE \Big\|\langle \mathcal{M}^{(m-1)}_{\eta_L(s)},  \sigma(x,\cdot) \rangle  - \bE[\sigma(x,X_{\eta_L(s)})] \Big\|^2 \mu^{\bar{Z}^L}_{\eta_L(s)} (dx) \bigg] \,ds \bigg).  \label{eq:c1ineq2interact}
\end{eqnarray}
By Lemma \ref{lm:varianceinterpvsnoninterpinteract}, there exists a constant $c$ such that
\begin{equation} \label{eq:c1ineq5interact}
	\begin{split}
	&	\int_{\bR^d} \bE \bigg[ \bigg( \bE\bigg[\langle \mathcal{M}^{(m)}_{\eta_L(t)},  P(x,\cdot) \rangle  \bigg|\F^{m-1} \bigg] -\langle \mathcal{M}^{(m)}_{\eta_L(t)},  P(x,\cdot) \rangle \bigg)^2\bigg] \,  \mu^{\bar{Z}^L}_{\eta_L(t)} (dx) \\ 
		& = \int_{\bR^d} \bE \bigg[ \text{Var} 
\bigg( \langle \mathcal{M}^{(m)}_{\eta_L(t)},  P(x,\cdot) \rangle \bigg|\F^{m-1} \bigg) \bigg] \,  \mu^{\bar{Z}^L}_{\eta_L(t)} (dx)  \leq c\sum_{\ell=0}^L \frac{ h_\ell}{N_{m,\ell}}.
	\end{split}
\end{equation}
Combining \eqref{eq:c1ineq1interact}, \eqref{eq:c1ineq2interact} and \eqref{eq:c1ineq5interact} yields the result.
\end{proof}

The complete algorithm
 consists of a sequence of nested MLMC estimators $\Big\{\langle \mathcal{M}^{(m)} , P(x,\cdot) \rangle \Big\}_{m=1,\ldots,M}$ and its error analysis is presented in Theorem \ref{lm:mseboundfinalinteract}. Note that we iterate the algorithm by replacing $P$ by the component real-valued functions $ \{ b_i \}_{1 \leq i \leq d}$ and $\{ \sigma_{i,j}  \}_{1 \leq i \leq d, 1 \leq j \leq r}$.

\subsection{Proof of Theorem \ref{lm:mseboundfinalinteract}} \label{sec mseboundfinalinteract}
 
\begin{proof}
First, the assumption that $Y^{i,0,\ell} = X_0$ gives
\begin{align}\label{eq:lmstartinglevel1interact}
	\sup_{0\leq t\leq T}\int_{\R^d}\bE\bigg[\Big| \bE[b(x,X_{\eta_L(t)})]- \langle \mathcal{M}^{(0)}_{\eta_L(t)},  b(x,\cdot) \rangle   \Big|^2 \quad \quad \quad \quad \quad \quad \quad   \\
	 \quad \quad \quad \quad \quad \quad  \quad \quad  + \Big\| \bE[\sigma(x,X_{\eta_L(t)})]-\langle \mathcal{M}^{(0)}_{\eta_L(t)},  \sigma(x,\cdot) \rangle  \Big\|^2\bigg]\mu^{\bar{Z}^L}_{\eta_L(t)}(dx) \leq c. \nonumber 
\end{align} 
Fixing $M>0$ and $P \in C^2_b (\bR^d)$,  we set 
\begin{equation}\label{eq:lmasequencesettinginteract}
	a_t^{(m)} \defeq\begin{cases}
       & \bE\bigg[\bigg( \langle \mathcal{M}^{(m)}_{\eta_L(t)},  P \rangle  - \bE[P(X_{\eta_L(t)})] \bigg)^2\bigg],\quad m=M,\\
       &\displaystyle\displaystyle \int_{\R^d}\bE\bigg[ \Big|\langle \mathcal{M}^{(m-1)}_{\eta_L(t)},  b(x,\cdot) \rangle- \bE[b(x,X_{\eta_L(t)})] \Big|^2\\
       &\hspace{1cm}+\Big\| \langle \mathcal{M}^{(m-1)}_{\eta_L(t)},  \sigma(x,\cdot) \rangle  - \bE[\sigma(x,X_{\eta_L(t)})] \Big\|^2 \bigg] \mu^{\bar{Z}^L}_{\eta_L(t)}(dx), \quad m \leq M-1. \\
\end{cases} 
\end{equation}
From Lemma \ref{lm:c1nonintinteract}, we observe that
\begin{align}\label{eq:lmrecurrent}
	a_t^{(m)} \leq c\bigg(b^{(m)} + \int_0^t a_s^{(m-1)}ds\bigg),\quad \quad \forall m \in \{ 1,2,\ldots, M \},
\end{align}
where $b^{(m)}=h_L^2+\sum_{\ell=0}^L\frac{h_\ell}{N_{m,\ell}}.$ Then one can easily show that 
\begin{align}\label{eq:lmrecurrent1interact}
	\sup_{0\leq t\leq T}a_t^M &\leq \sum_{m=0}^{M-1}b^{(M-m)}\dfrac{(cT)^{m}}{m!}+\bigg(\sup_{0\leq s\leq T}a_s^{(0)} \bigg)\cdot\frac{(cT)^M}{M!}.
\end{align}
Inequalities \eqref{eq:lmstartinglevel1interact} and \eqref{eq:lmrecurrent1interact} conclude the proof.
\end{proof}
We are now in a position to present the complexity theorem for iterated MLMC estimators of $\{\bE[P(X_{\eta_L(t)})]\}_{t\in[0,T]}$.
\begin{theorem}\label{thm:c2nonint}
Assume  \hkreg \, and \hLaw .  Fix $M>0$ and let $P \in C^2_b (\bR^d)$.  Then there exists some constant $c>0$  (independent of the choices of $M$, $L$ and $\{N_{m,\ell}\}_{m,\ell}$) such that for any $\eps<e^{-1}$, there exist $M$, $L$ and $ \{N_{m,\ell}\}_{m,\ell}$ such that for every $t\in[0,T]$,
\[
	MSE_{\eta_{L}(t)}^{(M)} (P) \defeq  \bE\bigg[\Big( \langle \mathcal{M}^{(M)}_{\eta_L(t)},  P \rangle  - \bE[P(X_{\eta_{L}(t)})] \Big)^2\bigg] \leq c \, \eps^2,
\]
and computational complexity is of the order $\eps^{-4}|\log\eps|^{3}$.
\end{theorem}
\begin{proof}

The cost of obtaining $\langle \mathcal{M}^{(M)}_{\eta_L(t)},  P \rangle$ involves $M$ iterations. In each iteration, one performs the standard MLMC algorithm, where the cost of approximating the law in the drift and diffusion coefficients is $\sum_{\ell'=0}^{L} N_{m-1,\ell'}$. Hence the overall cost $C:=C( M,L, \{N_{m,\ell}\}_{m,\ell} )$ of the algorithm is  
\begin{equation}\label{eq:constraintmain1interact2}
   C
	= \textstyle \sum_{\ell=0}^{L} h_\ell^{-1} N_{1,\ell}+\sum_{ m=2}^M \sum_{\ell=0}^{L} \bigg( h_\ell^{-1} N_{m,\ell}\sum_{\ell'=0}^{L} N_{m-1,\ell'} \bigg). 
\end{equation}
For convenience, we use the notation $x\lesssim y$ to denote that there exists a constant $c$  such that $x\leq c\ y$.
We shall establish specific values $M^{*},L^{*}, \{N^{*}_{m,\ell}\}_{m,\ell}$ (depending on $\epsilon$) such that the mean-square error satisfies \begin{equation} \textstyle	\sum_{m=1}^{M^{*}}\frac{c^{M^{*}-m}}{(M^{*}-m)!}\big(h_{L^{*}}^2+\sum_{\ell=0}^{L^{*}}\frac{h_\ell}{N^{*}_{m,\ell}} \big)+\frac{c^{M^{*}-1}}{M^{*}!}\lesssim\eps^2\label{eq:constraintcom1interact2}
\end{equation}
and show that corresponding computational complexity is of order $\eps^{-4}|\log\eps|^{3}$.
Firstly, we define 
\begin{equation}\label{eq:defMstar}
    M^{*}:= \left \lfloor \log(\eps^{-1}) \right \rfloor
    \implies c^{M^{*}-1} (M^{*}!)^{-1} \lesssim \epsilon^2\,
\end{equation}  
by Stirling's approximation. For $m \in \{1, \ldots, M^*\}$, we define $\eps_m^2\defeq w_m \eps^2$, for some sequence $\{w_m\}_{m=1}^{M^*}$ (depending on $M^*$ and $\epsilon$) which satisfies the following conditions:
\begin{itemize}
\item[(C{1})] Minimum condition: For each $m$, $w_m\geq w_{M^*} = 1$;
\item[ (C{2})] Weight condition: $\sum_{m=1}^{M^{*}}\frac{c^{M^{*}-m}}{(M^{*}-m)!}w_m \leq K$;
\item[(C{3})] Cost condition: $  \sum_{m=1}^{M^{*}} w^{-1}_m \leq K$,
\end{itemize}
for some constant $K>0$. (See Lemma \ref{lm:verficationepsm} for a concrete example.) 
Subsequently, we define 
\begin{align}\label{eq:conepsm1}    
 L^* := \max_{1 \leq m \leq M^*} L^{*}_m,
\quad 
 L^{*}_m:= \begin{cases} 
      \big| \left \lfloor \log(\eps_m^{-1}) \right \rfloor \big|, & {\eps}_m \leq e, \\
      1, & {\eps}_m > e. \\
         \end{cases}
\end{align}
We also define
\begin{align}\label{eq:complexityNkl}
N^{*}_{m,\ell}:=  \left \lceil  \eps_{m}^{-2} (L^{*}+1)h_\ell \right \rceil, \quad \ell \in \{ 0,\ldots, L^{*} \}, \quad m \in \{ 1, \ldots, M^* \}. 
\end{align}
Note that $ h_{L^{*}} \lesssim \eps_m$,  for any $m \in \{1, \ldots, M^*\}$. To see this, we show that $h_{L^*_m} \lesssim \eps_m$ by considering the following three cases.
\begin{enumerate}
    \item Case I: $\eps_m > e$. In this case,
    $$ h_{L^*_m} = T2^{-L^*_m} = T2^{-1} = \Big( \frac{ T2^{-1}}{e} \Big) e < \Big( \frac{ T2^{-1}}{e} \Big) \eps_m.$$ 
    \item Case II: $1 \leq \eps_m \leq e$. In this case,
    $$ h_{L^*_m} = T2^{-L^*_m} = T2^{\left \lfloor \log ( \eps^{-1}_m) \right \rfloor} =T2^{- \log ( \eps_m) } \leq T \leq T \eps_m. $$ 
    \item Case III: $0 < \eps_m < 1$. Without loss of generality, we assume that $T \leq \frac{1}{2}$. (We can scale $T$ by an appropriate factor if it is greater than $\frac{1}{2}$.) In this case,
    $$ \log(\eps_m) \leq \bigg( \frac{1}{\log 2} \bigg) \log(\eps_m) - \frac{ \log (2T)}{\log 2} = \frac{\log (\frac{\eps_m}{2T})}{\log 2} = \log_2 \Big( \frac{\eps_m}{2T} \Big),$$
    which implies that
    $$ h_{L^*_m} = T2^{-L^*_m} = T2^{- \left \lfloor \log (\eps^{-1}_m) \right \rfloor} \leq T2^{-   \big( \log (\eps^{-1}_m) -1 \big)} = 2T 2^{ \log (\eps_m)} \leq \eps_m.$$ 
\end{enumerate}
We can therefore observe that
\begin{eqnarray}
 && \sum_{m=1}^{M^{*}}\frac{c^{M^{*}-m}}{(M^{*}-m)!}\bigg(h_{L^{*}}^2 +\sum_{\ell=0}^{L^{*}}\frac{h_\ell}{N^{*}_{m,\ell}} \bigg) \nonumber \\
 & \leq & \sum_{m=1}^{M^{*}}\frac{c^{M^{*}-m}}{(M^{*}-m)!}\bigg(h_{L^{*}}^2 +\sum_{\ell=0}^{L^{*}}\frac{h_\ell}{\eps_{m}^{-2} (L^{*}+1)h_\ell } \bigg) \nonumber \\
 & \lesssim & \sum_{m=1}^{M^{*}}\frac{c^{M^{*}-m}}{(M^{*}-m)!} \eps^2_m \lesssim \eps^2, \nonumber 
\end{eqnarray}
by property (C2). Combining this estimate with \eqref{eq:defMstar}, we conclude that the constraint \eqref{eq:constraintcom1interact2} is satisfied.

It remains to compute the complexity of the cost under the values $M^{*},L^{*}, \{N^{*}_{m,\ell}\}_{m,\ell}$.
\begin{eqnarray}
C&=& \sum_{\ell=0}^{L^{*}} \bigg( h_\ell^{-1} \left \lceil  \eps_{1}^{-2} (L^{*}+1)h_\ell \right \rceil \bigg) + \sum_{ m=2}^{M^{*}} \sum_{\ell=0}^{L^{*}} \bigg( h_\ell^{-1} \left \lceil  \eps_{m}^{-2} (L^{*}+1)h_\ell \right \rceil  \nonumber \\
&& \sum_{\ell'=0}^{L^{*}} \left \lceil  \eps_{m-1}^{-2} (L^{*}+1)h_{\ell'} \right \rceil \bigg) \nonumber \\
&\lesssim &\sum_{\ell=0}^{L^{*}} \bigg( h_\ell^{-1} \Big(  \eps_{1}^{-2} (L^{*}+1)h_\ell +1 \Big) \bigg) + \sum_{ m=2}^{M^{*}} \sum_{\ell=0}^{L^{*}} \bigg( h_\ell^{-1} \Big(  \eps_{m}^{-2} (L^{*}+1)h_\ell +1 \Big)  \nonumber \\
&&  \Big(   \eps_{m-1}^{-2} (L^{*}+1) + (L^{*}+1) \Big) \bigg) \nonumber \\
& \lesssim& \eps^{-2}(L^{*}+1)^2+ \sum_{m=2}^{M^{*}} \bigg(   \eps_m^{-2}\eps_{m-1}^{-2} (L^{*}+1)^3 + \eps_m^{-2} (L^{*}+1)^3+ \nonumber \\
&& \eps^{-1} (L^{*}+1)^2 \eps_{m-1}^{-2} +  \eps^{-1} (L^{*}+1)^2\bigg) \nonumber \\
& \lesssim& \eps^{-2}|\log(\eps^{-1})|^2+ |\log(\eps^{-1})|^3 \sum_{m=2}^{M^{*}}    \eps_m^{-2}\eps_{m-1}^{-2} + |\log(\eps^{-1})|^3\sum_{m=2}^{M^{*}} \eps_m^{-2}   \nonumber \\
&& + \eps^{-1} |\log(\eps^{-1})|^2 \sum_{m=2}^{M^{*}} \eps_{m-1}^{-2} +  \eps^{-1} |\log(\eps^{-1})|^2 M^*, \label{eq:complexitycostineq1}
\end{eqnarray}
where, we have used in the last two estimates the bounds $L^* \leq \log( \eps^{-1}) $ (by property (C1)) and 
$ h^{-1}_{\ell} = T^{-1} 2^{\ell} \leq T^{-1} 2^{L^{*}} \lesssim  2^{\log( \eps^{-1})} \lesssim \eps^{-1}$. 
Finally, by properties (C1) and (C3) of $\{w_m\}_{m=1}^{M^*}$, together with \eqref{eq:complexitycostineq1} and \eqref{eq:defMstar}, we conclude that
$ C \lesssim \eps^{-4} | \log(\eps)|^3.$
\end{proof}



\subsection{Non-interacting kernels} \label{sec:nonintk}

Here we remark how the theory developed in this work would simplify, if we only treated McKV-SDEs with non-interacting kernels given by 
\begin{equation}\label{eq:defxtm}
dX_t =b \bigg( X_t,\int_{\bR^d }f(y)\, \mu_{t}^{X}(dy) \bigg) \, dt +\sigma \bigg( X_t,\int_{\bR^d }g(y) \, \mu_{t}^{X}(dy) \bigg) \, dW_t, 
\end{equation}
for some continuous functions $b: \mathbb{R}^d \times \mathbb{R}^{q} \rightarrow \mathbb{R}^d$ and $\sigma: \mathbb{R}^d \times \mathbb{R}^q \rightarrow \mathbb{R}^{d\otimes r}$. We  assume \hkreg \, and \hLaw. We also assume that each component function of $f$ and $g$ belongs to the set $C_b^2(\mathbb{R}^d ,\mathbb{R}^{q})$. 
The corresponding MLMC particle system is
\[ 
 dY^{i,m,\ell}_t =  b \Big(  Y^{i,m,\ell}_{\eta_{\ell} (t)}, \langle \mathcal{M}^{(m-1)}_{\eta_{\ell}(t)},  f \rangle  \Big) \, dt + \sigma \Big(  Y^{i,m,\ell}_{\eta_{\ell} (t)}, \langle \mathcal{M}^{(m-1)}_{\eta_{\ell}(t)},  g \rangle \Big) \, dW^{i,m}_t.
\]
To study this case, we adopt the abstract framework with $\overline b(x,\mu) := b(x,\langle \mu, f\rangle) $,  $\overline \sigma(x,\mu) := \sigma(x,\langle \mu, g\rangle) $ and $\cV$ being defined as before. Clearly, this is a special case of  the equation studied so far and hence all the results apply. The main difference stems from the complexity analysis as the term $\sum_{\ell=0}^{L} h^{-1}_{\ell} N_{m, \ell} \sum_{\ell'=0}^{L}  N_{m-1, \ell'} $ in \eqref{eq:constraintmain1interact2} is replaced by $\sum_{\ell=0}^{L} h^{-1}_{\ell} N_{m, \ell} \, \, +$ \\ $\sum_{\ell'=0}^{L}  h^{-1}_{\ell'} N_{m-1, \ell'} $. By performing the same computation as in the proof of Theorem \ref{thm:c2nonint}, we can show that the computational complexity is reduced to the order of $\eps^{-2}|\log\eps|^2$.

\subsection{Plain iterated particle system}

The  proof of the following theorem constitutes a special case of Lemma \ref{lm:c1nonintinteract} and Theorem \ref{lm:mseboundfinalinteract}.
\begin{theorem}\label{lm:mseboundfinalinteractparticle}
Assume  \hkreg \, and \hLaw.  Fix $M>0$ and let $P \in C^2_b (\bR^d)$.  We define the mean-square error as
\[
MSE_t^{(M)}(P) \defeq \bE\bigg[\bigg( \frac{1}{N_{M}}\sum_{i=1}^{N_{M}}P({\overline{Y}}^{i,M}_t)- \bE[P(X_{t})] \bigg)^2\bigg].
\]
Then for every $t\in[0, T]$,
\[
	MSE_{\eta(t)}^{(M)} (P)\leq c\bigg\{h^2+\sum_{m=1}^{M}\frac{c^{M-m}}{(M-m)!}\cdot \frac{1}{N_{m}} +\frac{c^{M-1}}{M!}\bigg\},
\]
for some constant $c>0$ that does not depend on $M$ or $N_1, \ldots, N_M$.
\end{theorem}
The following theorem concerns the computational complexity in the estimation of $\{\bE[P(X_{\eta(t)})]\}_{t\in[0,T]}$, whose proof follows similar procedures as the proof of Theorem \ref{thm:c2nonint} and is omitted.
\begin{theorem}\label{thm:c2nonintmc}
Assume  \hkreg \,  and \hLaw .  Fix $M>0$ and let $P \in C^2_b (\bR^d)$. Then there exists some constant $c>0$ $($independent of the choices of $M$ and $\{N_{m}\}_{1 \le m\le M})$ such that  for any $\eps<e^{-1}$, there exist $M$ and $ \{N_{m}\}_{0\le m\le M}$ such that for every $t\in[0,T]$,
\begin{equation} \label{eq mse M}
	MSE_{\eta(t)}^{(M)} (P) \defeq  \bE\bigg[(\frac{1}{N_{M}}\sum_{i=1}^{N_{M}}P({\overline{Y}}^{i,M}_{\eta(t)}) - \bE[P(X_{\eta(t)})])^2\bigg] \leq c \eps^2,
\end{equation}
and computational complexity $C$ is of the order $\eps^{-5}$.
\end{theorem}



\section{Numerical results}
In this section, we present numerical simulations that confirms that iterative MLMC method achieves one order better computational complexity comparing to classical particle system. Furthermore, numerical experiments indicate that the iterative MLMC method works well even if the coefficients of the McKV-SDEs do not satisfy previously stated regularity and growth assumptions.   We compare the following methods
\begin{itemize}
\item Classical particle system \eqref{eq:sEuler2}, 
\item MC Picard \rom{1} -  \emph{iterative particle system} \eqref{eq:xminteract2} with fixed number of particles $N$ for all Picard steps,
\item  MC Picard \rom{2} -  \emph{iterative particle system} \eqref{eq:xminteract2} with an increasing sequence of particles $\{N_m\}_{m=1,\ldots,M}$  where $N_m  = w_mN_M $ (see the choice of $w_m$ in Lemma \ref{lm:verficationepsm}),
\item Iterated MLMC particle system outlined in Algorithm \ref{Algorithminteract}. 
\end{itemize} 


\subsection{Kuramoto model}
First, we provide a numerical example of a one-dimensional stochastic differential equation derived from the Kuramoto model: 
\begin{equation*}
\begin{split}
		dX_t &=\int_\R\sin(X_t-y)\mu^X_t(dy) \, dt + dW_t, \quad \quad   t\in[0,1], \quad \quad X_0=0, \\
		&= \sin(X_t)\int_\R\cos(y)\mu^X_t(dy) - \cos(X_t)\int_\R\sin(y)\mu^X_t(dy) \, dt + dW_t\,.
	\end{split}
\end{equation*}
For the numerical tests we work with the the bottom representation.  We set $P(x)=\sqrt{1+x^2}$. For the initial condition of the iterative algorithm we choose  $Y^{0,\ell}_t\sim N(0,t)$. 

\begin{figure}[!h]
\centering
\begin{subfigure}{.49\textwidth}
  \centering
  \includegraphics[width=.9\linewidth]{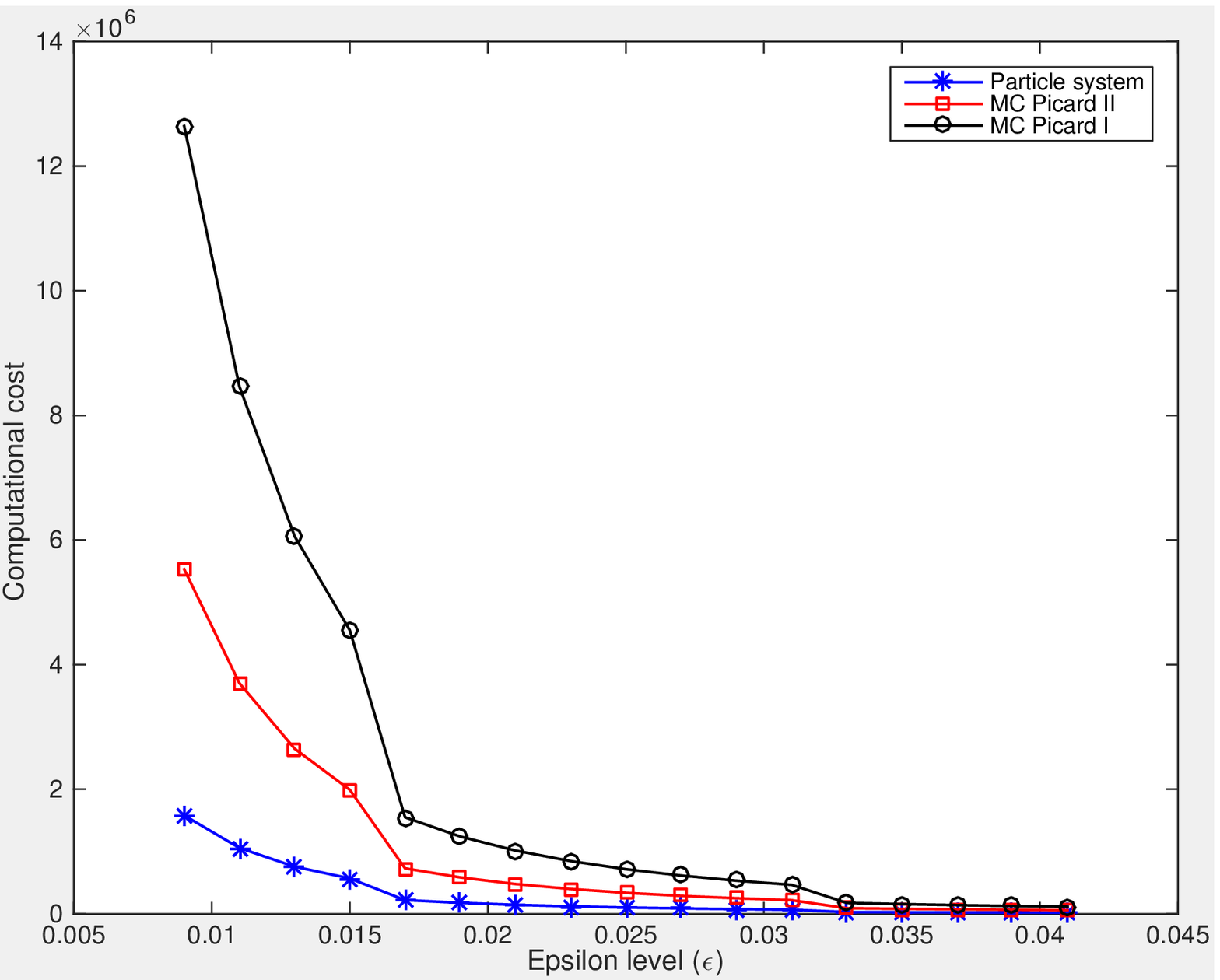}
  \caption{Iterated MC  vs Particle method}
  \label{fig:noninteractingexample11a}
\end{subfigure}
\begin{subfigure}{.49\textwidth}
  \centering
 \includegraphics[width=.9\linewidth]{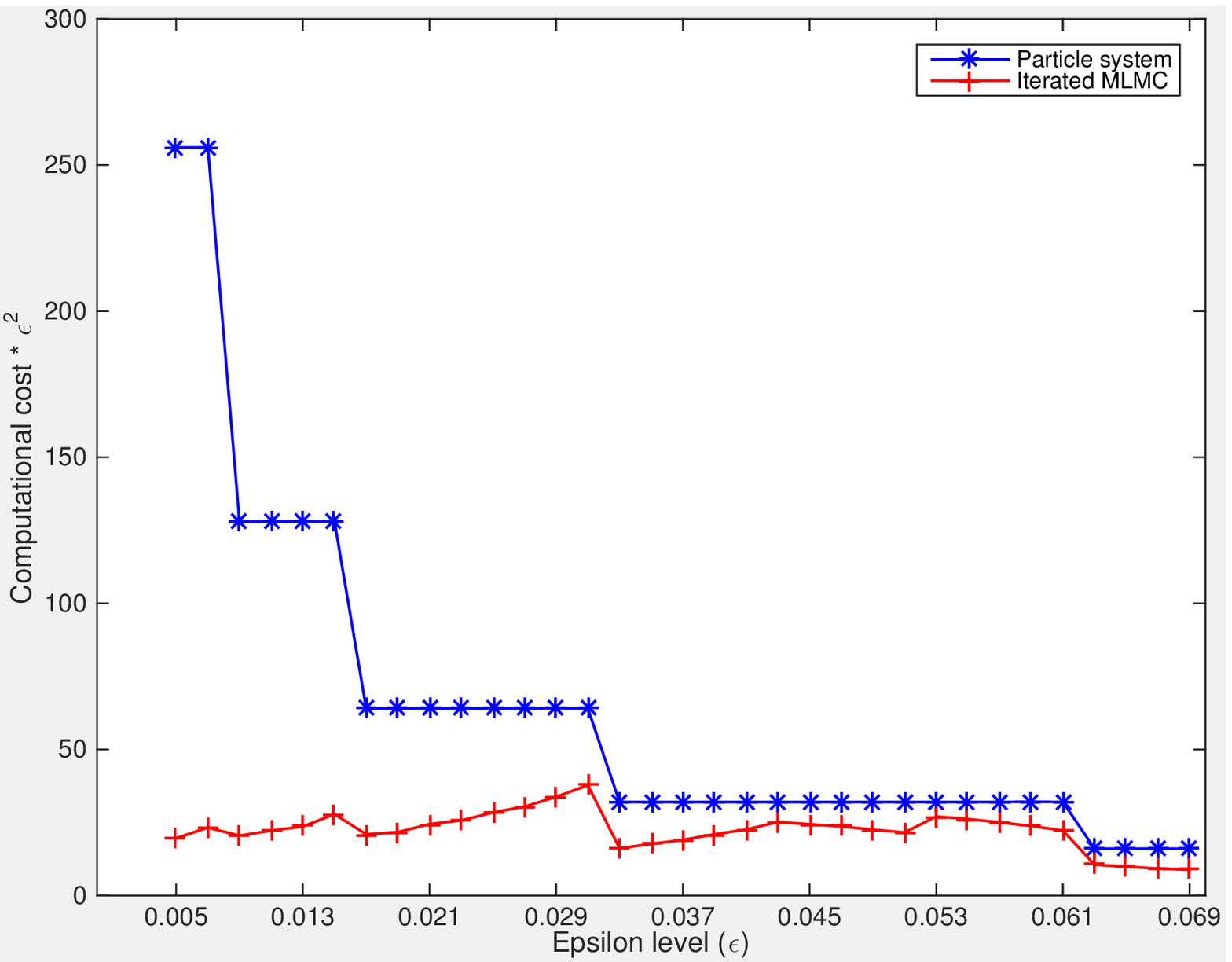}
 \caption{Iterated MLMC vs Particle method}
  \label{fig:noninteractingexample11c}
\end{subfigure}
\centering
\begin{subfigure}{.49\textwidth}
  \centering
  \includegraphics[width=.9\linewidth]{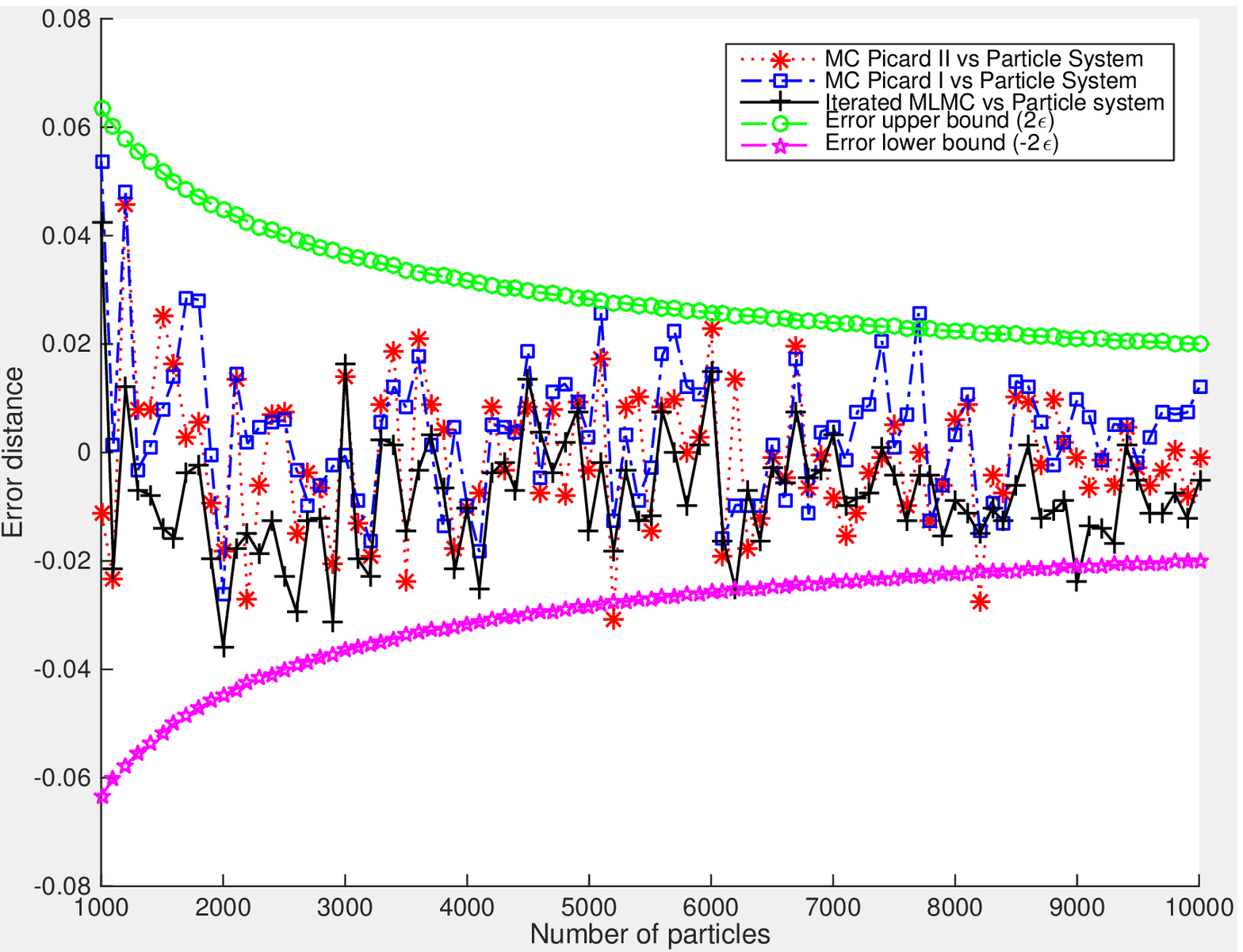}
  \caption{Approximation error}
  \label{fig:noninteractingexample11d}
\end{subfigure}%
\begin{subfigure}{.49\textwidth}
  \centering
  \includegraphics[width=.9\linewidth]{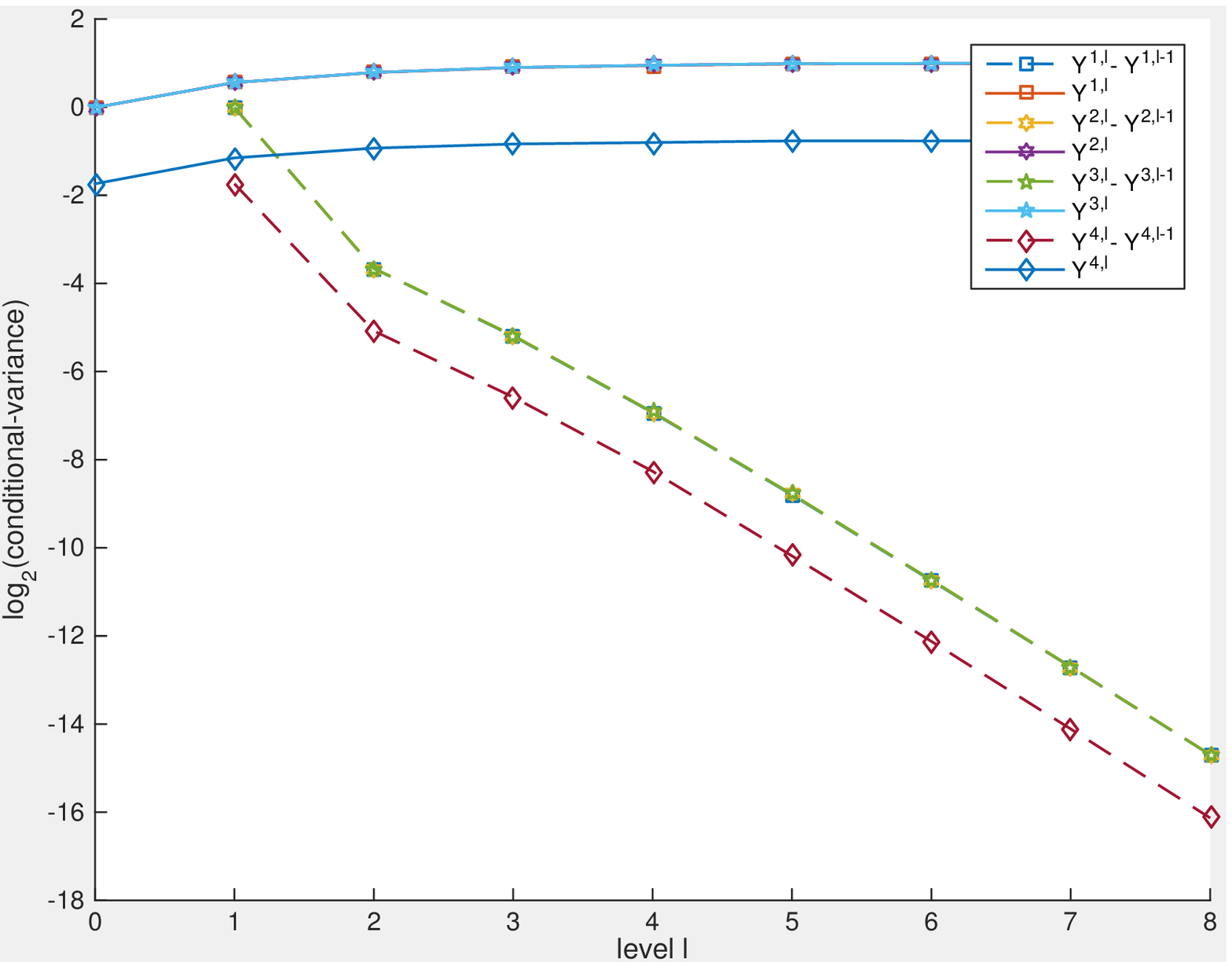}
  \caption{Variance of $\log_2$ against $\ell$ for all Picard steps}
  \label{fig:log2variancenonint}
\end{subfigure}%
\caption{Result of Kuramoto model}
\label{fig:noninteractingexample1}
\end{figure}
Figure \ref{fig:noninteractingexample11a} shows that both MC Picard \rom{1} and MC Picard \rom{2} are less efficient than the classical particle system. In Figure \ref{fig:noninteractingexample11c}, the iterated MLMC particle system achieves computational complexity of order $\eps^{-2}$ (note that here the cost of simulating particle system is $N$ per Euler step and not $N^2$  - see Section \ref{sec:nonintk}). 

 Figure \ref{fig:noninteractingexample11d} illustrates that the approximation error of iterated methods is within $2\eps$ of that of the classical particle system and that it decreases as number of particles increases. 
  
  Figure \ref{fig:log2variancenonint} depicts $\text{Var}[Y_T^{1,m,\ell}|\cM^{(m-1)}]$ and $\text{Var}[Y_T^{1,m,\ell}-Y_T^{1,m,\ell-1}|\cM^{(m-1)}]$ (in log scale) for each Picard step across levels $\ell$ . We see that that the conditional MLMC decays with rate $2$. This is higher than the rate given in Lemma \ref{lm:variance1nonintinteract}, since this example treats SDE with constant diffusion coefficient for which Euler scheme achieves higher strong convergence rate.

%
%
 
\subsection{Polynomial drift}
We consider the following McKV-SDE:
\begin{equation}\label{eq:defxtm2}
dX_t = (2X_t+\bE[X_t]-X_t\bE[X_t^2]) dt + X_t dW_t, \quad   t\in[0,1],  \quad X_0=1\,.
\end{equation}
Assumption \ref{as 1} is clearly violated. Note that
\begin{equation*}
\begin{split}
d\bE[X_t] &= (3\bE[X_t] - \bE[X_t]\bE[X_t^2]) dt\, \quad \bE[X_0]=1\\
d\bE[X^2_t] &= (5\bE[X^2_t]+ 2(\bE[X_t])^2 - (\bE[X_t^2])^2) dt\, \quad \bE[X^2_0]=1\,.
\end{split}
\end{equation*}
By solving the above system of ODEs with Euler scheme we obtain particle free approximation to the solution of \eqref{eq:defxtm2} that we use as a reference for iterative MLMC method. Figure \ref{fig:interactingexample11}, shows that the iterated MLMC achieves computational complexity of order $\eps^{-2}$. Figure \ref{fig:interactingexample12} indicates that the approximation error of iterated methods is within less than $2\eps$ of that of the reference value and that it decreases as number of particles increases.

 \begin{figure}[!h]
\centering
\begin{subfigure}{.49\textwidth}
  \centering
  \includegraphics[width=.9\linewidth]{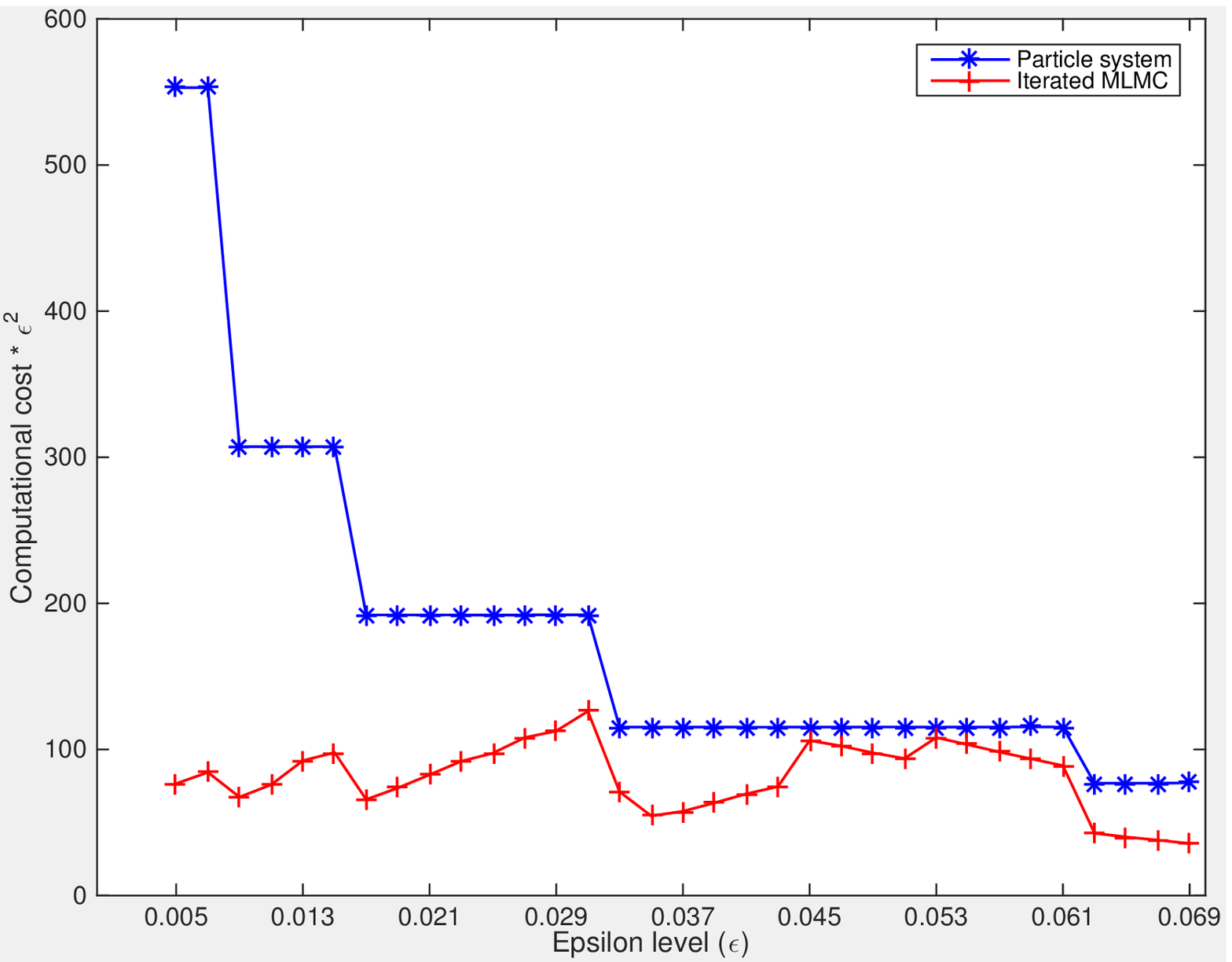}
  \caption{Iterated MLMC  vs Particle sytem}
  \label{fig:interactingexample11}
\end{subfigure}
\begin{subfigure}{.49\textwidth}
  \centering
 \includegraphics[width=.9\linewidth]{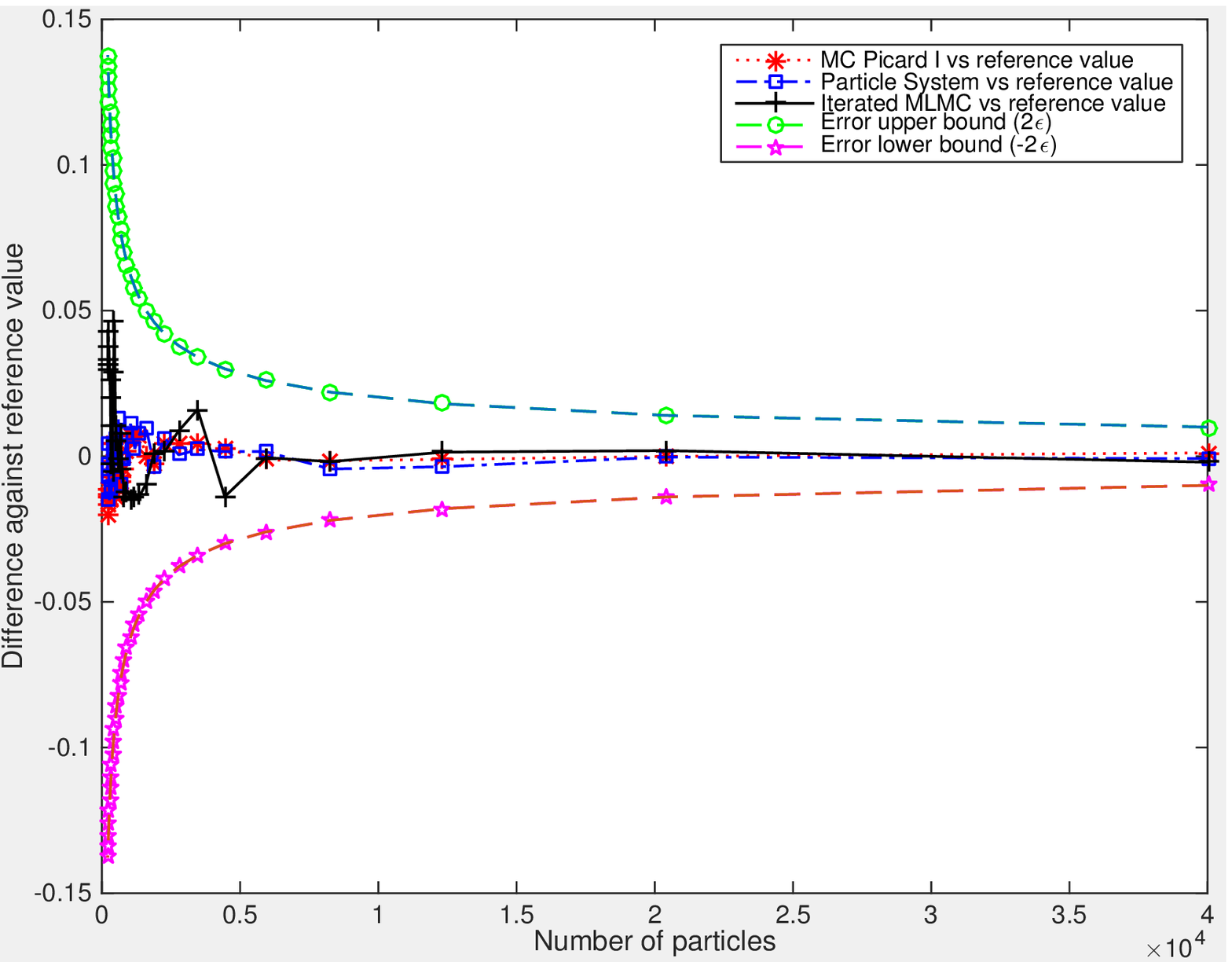}
 \caption{Approximation error}
  \label{fig:interactingexample12}
\end{subfigure}
\caption{Result of Polynomial drift}
\label{fig:interactingexample1}
\end{figure}

\subsection{Viscous Burgers equation}
Last, we perform a numerical experiment for the discontinuous case (not Lipschitz) corresponding to the Burgers equation (\cite{bossy1997comparison}) given by
\begin{equation}\label{eq:defxtm3}
dX_t = \bar{F}_t(X_t) dt + \dfrac{1}{4} dW_t, \quad \quad   t\in[0,1], \quad \quad X_0=0,
\end{equation}
where $\bar{F}_t(x) = \bP(X_t\geq x)$. Linking to the Fokker-Planck equation of $X_t$, it is important to notice that $\bar{F}_t(x)$ is the solution to the viscous Burgers equation:
\[
	\partial_t v (t,x) = \dfrac{1}{32}\partial_{xx} v(t,x) - v(t,x)\partial_x v(t,x).
\]
where $\bar{F}_0(x) = \1_{\{x\leq 0\}}$ since the initial condition $X_0=0$. The Cole-Hopf transformation results in, for any $t\in(0,1]$
\[
	\bar{F}_t(x) = \dfrac{\mathcal{N}(\dfrac{4t-4x}{\sqrt{t}})}{\exp(16x-8t)\mathcal{N}(\dfrac{4x}{\sqrt{t}})+\mathcal{N}(\dfrac{4t-4x}{\sqrt{t}})},
\]
where $\mathcal{N}(x)= \int_{-\infty}^x\exp(\dfrac{-y^2}{2})\dfrac{dy}{\sqrt{2\pi}}$. Then we take $\bar{F}_1(0.5)=0.5$ as the reference value.
 \begin{figure}[!h]
\centering
\begin{subfigure}{.49\textwidth}
  \centering
  \includegraphics[width=.9\linewidth]{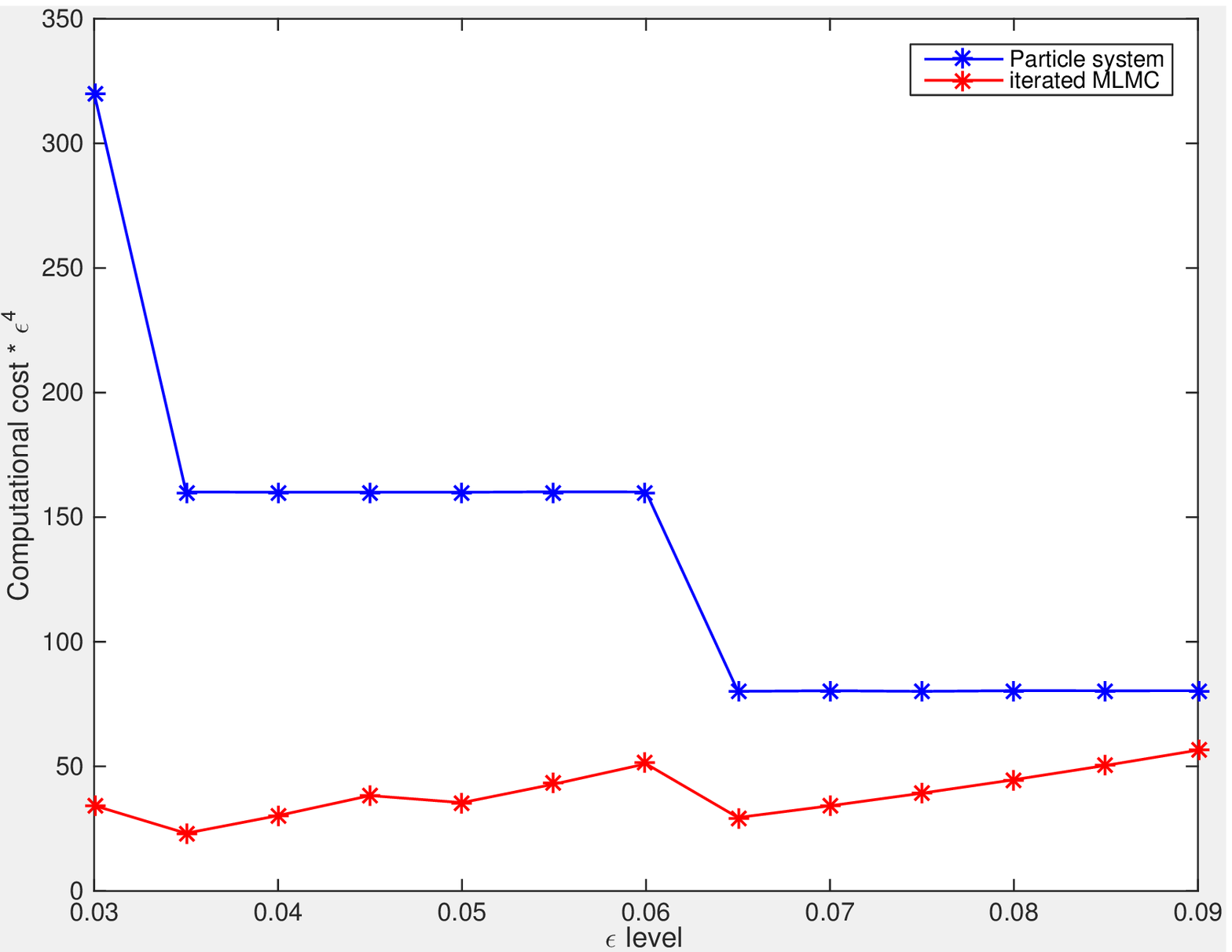}
  \caption{Iterated MLMC  vs Particle sytem}
  \label{fig:interactingexample21}
\end{subfigure}
\begin{subfigure}{.49\textwidth}
  \centering
 \includegraphics[width=.9\linewidth]{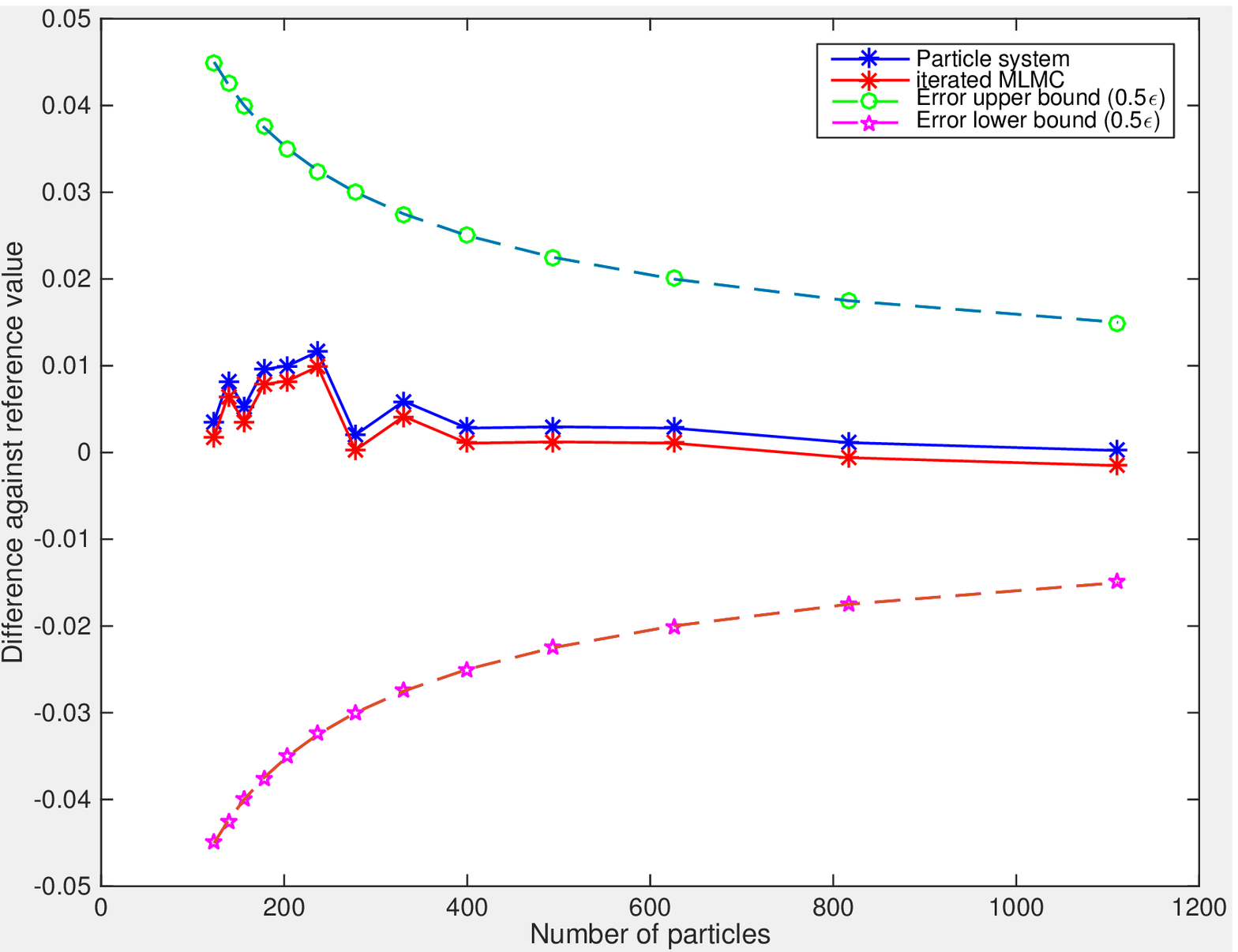}
 \caption{Approximation error}
  \label{fig:interactingexample22}
\end{subfigure}
\caption{Result of viscous Burgers equation}
\label{fig:interactingexample2}
\end{figure}
In Figure \ref{fig:interactingexample21}, the iterated MLMC achieves computational complexity of order $\eps^{-4}$. Figure \ref{fig:interactingexample22} demonstrates the similar desired behaviour of the approximation error as observed in the case of the polynomial drift.

\appendix 

\section{Proofs and useful lemmas} \label{sec intergrabilitySDErandomcoeffinteract}
\begin{proof}[Proof of Lemma \ref{lm:intergrabilitySDErandomcoeffinteract} ]

Given any $\ell$, let us define a sequence of stopping times  $\t_M := \inf\{t\geq0: |Z_t^{\ell} - Z^{\ell}_0| \geq M\}.$ For any $t\in[0,T]$, we consider the stopped process $Z_{t\wedge \t_M}^{\ell}$ and compute by the Burkholder-Davis-Gundy and H\"older inequalities and assumptions \hvLip \, and \hLaw \, to obtain that 
\begin{eqnarray}
\bE\bigg[\sup_{0\leq u\leq t}|Z_{u\wedge \t_M}^{\ell} |^p\bigg]& \leq &c\ \bigg(\bE[|Z^{\ell}_0|^p] + t^{p-1}\bE\bigg[\int_0^t|\overline{b}(Z_{\eta(s)\wedge\t_M}^\ell, \cV_{\eta(s)} )|^p ds\bigg] \nonumber\\
&& +t^{\frac{p}{2}-1}\bE\bigg[\int_0^t\norm{\overline{\sigma}(Z_{\eta(s) \wedge\t_M }^\ell, \cV_{\eta(s)} )}^pds\bigg] \bigg). \nonumber \\
& \leq & \bigg( 1+\bE\bigg[\int_0^t \bigg| \int_\Rd|y|^p\cV_{\eta(s)}(dy) \bigg| \, ds \bigg] \nonumber \\
&& + \int_0^t\bE \bigg[\sup_{0\leq u\leq s}|Z_{u\wedge \t_M}^{\ell}|^p \bigg]\, ds \bigg). \nonumber \end{eqnarray}
Note that, by \hLaw, 
$$ \bE \bigg[\sup_{0\leq u\leq s}|Z_{u\wedge \t_M}^{\ell}|^p \bigg] \leq c \bigg( \bE \bigg[\sup_{0\leq u\leq s}|Z_{u\wedge \t_M}^{\ell} - Z_{0}^{\ell} |^p \bigg] + \bE |Z_{0}^{\ell} |^p \bigg)  < + \infty.$$
By Gronwall's lemma,	
 $$  \bE\bigg[\sup_{0\leq u\leq t}|Z_{u\wedge \t_M}^{\ell} |^p\bigg]  \leq c\ \bigg(1+\bE \bigg[\int_0^T \bigg|\int_\Rd|y|^p\cV_{\eta(s)}(dy) \bigg| \, ds \bigg] \bigg). $$
Furthermore,  since $\sup_{0\leq t\leq T}|Z_{t\wedge \t_M}^{\ell}|^p$
is a non-decreasing  sequence (in $M$) converging pointwise to $\sup_{0\leq t\leq T}|Z_t^{\ell}|^p$, the lemma follows from the monotone convergence theorem.

\end{proof}

\begin{lemma} \label{lem:conmartingaleinteract}
Let $\{Q_t \}_{t \in [0,T]}$ be a cadlag square-integrable process adapted to the filtration $\{ \mathcal{F}_t \}_{t \in [0,T]}$. Suppose that $\{W_t \}_{t \in [0,T]}$ is a $\{ \mathcal{F}_t \}_{t \in [0,T]}$-Brownian motion. Let $\mathcal{G}$ be a $\sigma$-algebra such that $\mathcal{G} \subseteq \mathcal{F}_0$. Then the following equalities hold for any $t \in [0,T]$.
\begin{eqnarray}
 \text{(a)} && \bE \bigg[\int_{0}^t Q_s dW_s \, \bigg| \,  \cG \bigg] = 0,	\nonumber \\
 \text{(b)} &&  \bE \bigg[ \bigg( \int_0^t Q_s \, dW_s \bigg)^2 \bigg| \mathcal{G} \bigg] = \bE \bigg[ \int_0^t Q^2_s \,ds \bigg| \mathcal{G} \bigg]. \nonumber
 \end{eqnarray}
\end{lemma}
The proof follows from standard results of stochastic calculus and is omitted.


\begin{lemma}\label{lm:verficationepsm}
The sequence $\{w_m\}_{m=1}^{M^*}$ defined by 
\begin{align*}
	w_m\defeq \begin{cases}
\max \Big\{\frac{(M^*-m-2)!}{c^{M^*-m-2}},1 \Big\}, &\quad 1 \leq m\leq M^*-2,\\
	 1,&\quad M^*-1\leq m\leq M^*,
	\end{cases}
\end{align*}
satisfies properties $(C1)$ to $(C3)$ stipulated in the proof of Theorem \ref{thm:c2nonint}.
\end{lemma}
\begin{proof}
First, property (C1) follows easily from the definition of $w_m$. For property (C2), we verify that 
\begin{eqnarray*}
&& \sum_{m=1}^{M^{*}}\frac{c^{M^{*}-m}}{(M^{*}-m)!}w_m \nonumber \\
&\leq& \sum_{m=1}^{M^{*}-2}\frac{c^{M^{*}-m}}{(M^{*}-m)!} \Big(\frac{(M^{*}-m-2)!}{c^{M^{*}-m-2}}+1\Big)+\sum_{m=M^*-1}^{M^{*}}\frac{c^{M^{*}-m}}{(M^{*}-m)!}\\
&=&\sum_{m=1}^{M^{*}}\frac{c^{M^{*}-m}}{(M^{*}-m)!}+c^2\sum_{m=1}^{M^{*}-2}\frac{1}{(M^{*}-m)(M^{*}-m-1)}\\
&=& \sum_{m=1}^{M^{*}}\frac{c^{M^{*}-m}}{(M^{*}-m)!}+c^2 \bigg( 1-\frac{1}{M^{*}-1} \bigg) \leq e^c  + c^2.
\end{eqnarray*}
Lastly, we show this sequence satisfies property (C3). Indeed,
\[
	\sum_{m=1}^{M^*}w_m^{-1} = \sum_{m=1}^{M^{*}-2} \frac{c^{M^{*}-m-2}}{(M^{*}-m-2)!}+2 \leq e^{c}+2.
\]
\end{proof}

\section{Algorithm for the MLMC particle system}

\begin{algorithm}
\KwIn{Initial measure $\mu^0$ for $Y^{i,0,\ell}$, global Lipschitz payoff function $C^2_p \ni P:\R^d\rightarrow\R$ and accuracy level $\eps$}
\KwOut{ $\langle \mathcal\cM_{T}^{(M)},  P \rangle $, the approximation for our goal $\bE[P(X_T)]$.}
\nl Fix parameters $M$ (see \eqref{eq:defMstar}) and $L$ (see \eqref{eq:conepsm1}) that correspond to $\eps$\;   
\nl Given $\mu^0 =\text{Law}(Y^{i,0,0})$, sample $\{Y^{i,0,0}_{t_k^L}\}_{k=0,\ldots,2^{L}}$\;
\nl \For{ $m = 1$ to $M-1$}{
	\nl During $m$th Picard step, given samples $\{Y^{i,m-1,\ell}_{t_k^\ell}\}_{k=0,\ldots,2^{\ell}}^{\ell=0,\ldots,L}$, take \eqref{eq:contiousYinteract} and run MLMC to obtain $\{Y^{i,m,\ell}_{t_k^\ell}\}_{k=0,\ldots,2^{\ell}}^{\ell=0,\ldots,L}$. This requires calculating 
	\[
	\bigg( \langle \mathcal\cM_{t_0^L}^{(m-1)},  b(x,\cdot) \rangle ,\ldots, \langle \mathcal\cM_{t_{2^{L}}^L}^{(m-1)},  b(x,\cdot)  \rangle \bigg), 
	\]  
	\[
	\bigg( \langle \mathcal\cM_{t_0^L}^{(m-1)},  \sigma(x,\cdot) \rangle ,\ldots, \langle \mathcal\cM_{t_{2^{L}}^L}^{(m-1)},  \sigma(x,\cdot)  \rangle \bigg), 
	\]  
	where in place of $x$, we put particles $\{Y^{i,m,\ell}_{t_k^\ell}\}_{k=0,\ldots,2^{\ell}-1}^{\ell=0,\ldots,L}$\;}
\nl Given samples $\{Y^{i,M-1,\ell}_{t_k^\ell}\}_{k=0,\ldots,2^{\ell}}^{\ell=0,\ldots,L}$, run standard MLMC (with interpolation) to obtain the final vector of approximations  $\bigg( \langle \cM_{t_0^L}^{(M)}, P \rangle,\ldots, \langle \cM_{t_{2^{L}}^L}^{(M)}, P \rangle \bigg)$\;
\nl {\bf Return} $\langle \cM_{T}^{(M)}, P \rangle $.
    \caption{{\bf Nested MLMC with Picard scheme} \label{Algorithminteract}}

\end{algorithm}

\section*{Acknowledgements}
We are grateful to Mireille Bossy, Mike Giles and David \u{S}i\u{s}ka for helpful comments.

\bibliographystyle{abbrv}
\bibliography{Particles.bib}
\end{document}